\documentclass[reqno]{amsart}

\usepackage{enumerate}
\usepackage{comment}
\usepackage{longtable}
\usepackage{empheq}
\usepackage{amsmath}
\usepackage{amssymb}
\usepackage{array}
\usepackage{graphicx}
\usepackage{tikz}
\usepackage{pgf}
\usepackage{pgfpages}
\usepackage{comment}
\usetikzlibrary{decorations.pathreplacing}
\usetikzlibrary{calc}
\usetikzlibrary{decorations}

\usetikzlibrary{matrix}
\tikzstyle{vertex}=[circle,thin,draw=black!100,fill=black!100, inner sep=0pt, minimum width=4pt]
\tikzstyle{vertexg}=[circle, draw=black, inner sep=1pt, style=densely dotted, minimum width=4pt]
\tikzstyle{vertexinf}=[circle,thin,draw=black!100,fill=white!100, inner sep=0pt, minimum width=10pt]
\tikzstyle{pedge}=[draw=black!100,-]
\tikzstyle{nedge}=[draw=black!100,densely dashed]
\tikzstyle{gedge}=[draw=black!100,densely dotted]

\newtheorem{theorem}{Theorem}
\newtheorem{proposition}[theorem]{Proposition}
\newtheorem{lemma}[theorem]{Lemma}
\newtheorem{corollary}[theorem]{Corollary}
\newtheorem{definition}{Definition}

\begin{document}

\title[$1$-Salem graphs]{A classification of all $1$-Salem graphs}
\author{Lee Gumbrell}
\address{Department of Mathematics, Royal Holloway, University of
London, Egham Hill, Egham, Surrey, TW20 0EX, England, UK.}
\email{Lee.Gumbrell.2009@rhul.ac.uk}
\author{James McKee}
\address{Department of Mathematics, Royal Holloway, University of
London, Egham Hill, Egham, Surrey, TW20 0EX, England, UK.}
\email{James.McKee@rhul.ac.uk}


\begin{abstract}
One way to study certain classes of polynomials is by considering examples that are attached to combinatorial objects.
Any graph $G$ has an associated reciprocal polynomial $R_G$, and with two particular classes of reciprocal polynomials in mind one can ask the questions: (a) when is $R_G$ a product of cyclotomic polynomials (giving the \emph{cyclotomic graphs})? (b) when does $R_G$ have the minimal polynomial of a Salem number as its only non-cyclotomic factor (the non-trival \emph{Salem graphs})?
Cyclotomic graphs were classified by Smith in 1970.
Salem graphs are `spectrally close' to being cyclotomic, in that nearly all their eigenvalues are in the critical interval $[-2,2]$.
On the other hand Salem graphs do not need to be `combinatorially close' to being cyclotomic: the largest cyclotomic induced subgraph might be comparatively tiny.

We define an $m$-Salem graph to be a connected Salem graph $G$ for which $m$ is minimal such that there exists an induced cyclotomic subgraph of $G$ that has $m$ fewer vertices than $G$.
The $1$-Salem subgraphs are both spectrally close and combinatorially close to being cyclotomic.
Moreover, every Salem graph contains a $1$-Salem graph as an induced subgraph, so these $1$-Salem graphs provide some necessary substructure of all Salem graphs.  The main result of this paper is a complete combinatorial description of all $1$-Salem graphs: there are $26$ infinite families and $383$ sporadic examples.
\end{abstract}

\maketitle

\section{Motivation}\label{S:motivation}
\subsection{The remarkable interval $[-2,2]$}\label{SS:interval}

An algebraic integer $\theta$ is called \emph{totally real} if all its Galois conjugates are real.
One can then define the real interval $I_\theta = [\theta_{min},\theta_{max}]$, where $\theta_{min}$ and $\theta_{max}$ are the smallest and largest of the Galois conjugates of $\theta$.

The number $4$ plays a peculiar role in the theory of totally real algebraic integers.
If $J$ is any interval of length strictly less than $4$, then there exist only finitely many totally real $\theta$ such that $I_\theta\subseteq J$ (proved by Schur for $J$ centred on the origin, and extended to the general case by P\'{o}lya in a footnote to Schur's paper \cite{Sc1918}).
On the other hand, if $J$ is any interval of length strictly greater than $4$, then there exist infinitely many totally real $\theta$ such that $I_\theta\subseteq J$ (Robinson, \cite{Ro1962}).
If $J$ has length exactly $4$ and the endpoints are integers, then there are infinitely many such $\theta$.
If $J=[a,a+4]$ with $a$ an integer, and $I_\theta\subseteq J$, then $I_{\theta-a-2}\subseteq [-2,2]$, so that essentially the only such $J$ of interest is the special interval $[-2,2]$.
For other intervals of length $4$ nothing is known.

By a theorem of Kronecker \cite{Kr1857}, the only $\theta$ for which $I_\theta\subseteq[-2,2]$ are those of the form $\theta = \zeta+1/\zeta$ for $\zeta$ some (complex) root of $1$.
For this reason one sometimes refers to this as the \emph{cyclotomic} case.

\subsection{Algebraic integers from combinatorial objects}

One approach to the study of algebraic integers is to look for examples attached to combinatorial objects (graphs, knots, etc.).
The eigenvalues of graphs are totally real, and provide a fruitful hunting-ground for such algebraic integers.
Indeed Estes \cite{Es1992} showed that all totally real algebraic integers arise as eigenvalues of graphs (Hoffman \cite{Ho1972, Ho1975} having shown that the problem was equivalent to that for eigenvalues of integer symmetric matrices).
On the other hand it is easy to find examples of totally real $\theta$ for which the minimal polynomial of $\theta$ is not the characteristic polynomial of a graph, and with more work one can find examples where the minimal polynomial of $\theta$ is not even the \emph{minimal} polynomial of any integer symmetric matrix (\cite{Do2008, McK2010}).

If $G$ is a graph, we let $\chi_G$ denote its characteristic polynomial, and define $R_G$ to be the associated reciprocal polynomial $R_G(z):= z^{\deg \chi_G} \chi_G(z+1/z)$.
For any set $S$ of reciprocal polynomials of interest, there is a correspondingly interesting set of graphs $G$ for which $R_G\in S$.

\subsection{Cylotomic graphs}

After \S\ref{SS:interval}, the question of which graphs have all eigenvalues in the interval $[-2,2]$ acquires peculiar interest, and such graphs we call \emph{cyclotomic}.
(The reason for this name is that after Kronecker's theorem the associated reciprocal polynomial of such a graph is a product of cyclotomic polynomials.)
It is enough to consider connected examples, and Smith \cite{Sm1970} gave a complete classification.
The maximal connected examples comprise two infinite families and three sporadic cases: see Figure \ref{F:cyclograph}.

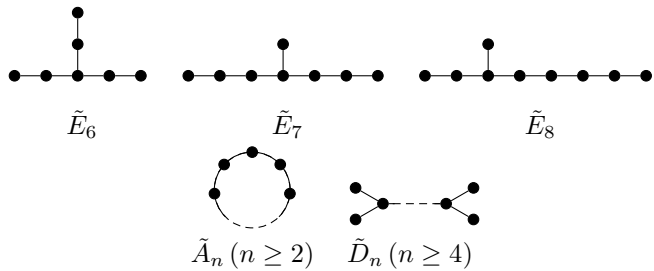
\begin{figure}[h]
\begin{center}
\begin{tabular}{ccc}
\begin{tikzpicture}[scale=0.42, auto]
\foreach \pos/\name in
{{(-2,0)/a},{(-1,0)/b},{(0,0)/c},{(1,0)/d},{(2,0)/e},{(0,1)/f},{(0,2)/g}}
\node[vertex] (\name) at \pos {};
\node[] at (2.2,-0.5) {};
\foreach \edgetype/\source/ \dest in {pedge/a/b,pedge/b/c,pedge/c/d,pedge/d/e,pedge/c/f,pedge/f/g}
\path[\edgetype] (\source) -- (\dest);
\end{tikzpicture}
&
\begin{tikzpicture}[scale=0.42, auto]
\foreach \pos/\name in
{{(-2,0)/a},{(-1,0)/b},{(0,0)/c},{(1,0)/d},{(2,0)/e},{(0,1)/f},{(-3,0)/g},{(3,0)/h}}
\node[vertex] (\name) at \pos {};
\node[] at (3.2,-0.5) {};
\foreach \edgetype/\source/ \dest in {pedge/a/b,pedge/b/c,pedge/c/d,pedge/d/e,pedge/c/f,pedge/a/g,pedge/e/h}
\path[\edgetype] (\source) -- (\dest);
\end{tikzpicture}
&
\begin{tikzpicture}[scale=0.42, auto]
\foreach \pos/\name in
{{(-2,0)/a},{(-1,0)/b},{(0,0)/c},{(1,0)/d},{(2,0)/e},{(0,1)/f},{(3,0)/g},{(4,0)/h},{(5,0)/i}}
\node[vertex] (\name) at \pos {};
\node[] at (5.2,-0.5) {};
\foreach \edgetype/\source/ \dest in {pedge/a/b,pedge/b/c,pedge/c/d,pedge/d/e,pedge/c/f,pedge/e/g,pedge/h/g,pedge/h/i}
\path[\edgetype] (\source) -- (\dest);
\end{tikzpicture}
\\
$\tilde{E}_6$ & $\tilde{E}_7$ & $\tilde{E}_8$
\end{tabular}
\begin{tabular}{cc}
\begin{tikzpicture}[scale=0.42, auto]
\draw[densely dashed] (0,0) circle (1.2cm);
\draw (1.2,0) arc (0:-45:1.2cm);
\draw (1.2,0) arc (0:225:1.2cm);
\foreach \pos/\name in
{{(0,1.2)/a},{(-0.888,0.807)/b},{(0.888,0.807)/c},{(-1.194,-0.115)/d},{(1.194,-0.115)/d}}
\node[vertex] (\name) at \pos {};
\end{tikzpicture}
&
\begin{tikzpicture}[scale=0.42, auto]
\foreach \pos/\name in
{{(0,0)/a},{(0.866,0.5)/b},{(0.866,-0.5)/c},{(-2,0)/aa},{(-2.866,0.5)/ab},{(-2.866,-0.5)/ac}}
\node[vertex] (\name) at \pos {};
\foreach \pos/\name in
{{(-0.667,0)/ax},{(-1.333,0)/ay}}
\node[] (\name) at \pos {};
\node[] at (-3.066,-0.5) {};
\foreach \edgetype/\source/ \dest in {pedge/a/b,pedge/c/a,pedge/a/ax,pedge/ay/aa,pedge/aa/ab,pedge/aa/ac,nedge/a/aa}
\path[\edgetype] (\source) -- (\dest);
\end{tikzpicture}
\\
$\tilde{A}_n \left(n \geq 2\right)$ & $\tilde{D}_n \left(n \geq 4\right)$
\end{tabular}
\caption{The maximal connected cyclotomic graphs.  Note that each graph has one more vertex than the subscript in its name.}
\label{F:cyclograph}
\end{center}
\end{figure}

\subsection{Salem numbers and Salem graphs: definitions}

A \emph{Salem number} is a real algebraic integer $\tau$ such that: (i) $\tau>1$; (ii) all Galois conjugates of $\tau$ other than $\tau$ itself lie in the unit disc $\{z\in\mathbb{C} : |z|\le 1\}$; (iii) at least of the Galois conjugates of $\tau$ has modulus exactly $1$.
It follows that the minimal polynomial of $\tau$ is reciprocal, $\tau$ has even degree at least $4$, and all of the Galois conjugates of $\tau$ other than $\tau$ and $1/\tau$ lie on the unit circle $\{z\in\mathbb{C} : |z|= 1\}$.

It is not known whether there are Salem numbers arbitrarily close to $1$.
The smallest known Salem number is $1.17628\dots$, the larger real root of the reciprocal polynomial $L(z) = z^{10} + z^9 - z^7 - z^6 - z^5 - z^4 - z^3 + z +1$.
This polynomial was discovered by considering examples attached to combinatorial objects: $L(-z)$ is the Alexander polynomial of a pretzel knot.

One can ask when a graph has as its reciprocal polynomial the minimal polynomial of a Salem number, and one might call these `Salem graphs'.
Since irreducibility of the reciprocal polynomial is not easy to capture combinatorially, it is natural to relax this definition slightly, and merely ask that the reciprocal polynomial has all but two roots on the unit circle.
Such graphs cannot be bipartite: the eigenvalues of a bipartite graph are symmetric about $0$, and the reciprocal polynomial as defined above cannot then have a unique root outside the unit disc.
The closest that a non-cyclotomic bipartite graph can get to being cyclotomic (measured in terms of the number of eigenvalues outside the critical interval $[-2,2]$) is to have a single root greater than $2$ and hence also a root less than $-2$.
In such cases we can exploit the symmetry of the eigenvalues to recover a polynomial whose roots are in the right place to be the minimal polynomial of a Salem number (but with no guarantee of irreducibility): we redefine $R_G(z) = z^{\deg(\chi_G)/2}\chi_G(\sqrt{z} + 1/\sqrt{z})$ for bipartite graphs.

This discussion motivates the following definition, which might otherwise seem unnecessarily convoluted.
We follow a common convention that if $G$ has $n$ vertices then its eigenvalues are labelled $\lambda_1\ge\lambda_2\ge\cdots\ge\lambda_n$.

\begin{definition}\label{D:Salem}
A bipartite graph $G$ is called a \emph{Salem graph} if the largest eigenvalue 
$\lambda_1$ is greater than $2$, and the remaining $n-1$ eigenvalues are no greater than 2. Then $\lambda_n=-\lambda_1<-2$.  A bipartite Salem graph is called \emph{trivial} if $\lambda_1^2\in \mathbb{Z}$.  The associated number $\tau\left(G\right)$ is the larger root of $\sqrt{z}+1/\sqrt{z}=\lambda_1$; this is a Salem number unless $G$ is trivial.

A non-bipartite graph $G$ is called a \emph{Salem graph} if the largest eigenvalue  $\lambda_1 > 2$ and the remaining $n-1$ eigenvalues are in the interval $\left[-2,2\right]$. A non-bipartite Salem graph is called \emph{trivial} if $\lambda_1\in \mathbb{Z}$.  The associated number $\tau\left(G\right)$ is the larger root of $z+1/z=\lambda_1$; this is a Salem number unless $G$ is trivial.
\end{definition}

This notion of a Salem graph was introduced in \cite{MS2005}.
Special cases of Salem graphs (or associated Coxeter transformations) and their connection with Salem numbers were studied in earlier (and later) papers: \cite{CW1992}, \cite{P1993}, \cite{MRS1999}, \cite{L1999}, \cite{M2000}, \cite{L2001}, \cite{M2002}, \cite{MS2004}, \cite{L2010}.
The highlight of \cite{MS2005} was a proof that all limit points of the set of graph Salem numbers are Pisot numbers, confirming a conjecture of Boyd \cite{B1977} in this setting.  Limit points of special subsets of the set of graph Salem numbers are also considered in \cite{MRS1999}, \cite{L1999}, \cite{M2000}, \cite{MS2004} and \cite{S2008}.  Many examples of Salem graphs and their spectra have been considered in other contexts, for example \cite{Z1989}, \cite{CDG1982}, \cite{MS2011}.

\subsection{The structure of Salem graphs: statement of the 
theorem}

A number of partial results concerning the structure of Salem graphs appear in \cite[Section 3]{MS2005}, and a complete description of all Salem trees in \cite[Theorem 7.2]{MS2005}, but a combinatorial description of all Salem graphs is still unknown.
In this paper we increase our understanding of the set of Salem graphs by considering how close they are to being cyclotomic in the following combinatorial sense: how many vertices do we need to delete to produce an induced cyclotomic subgraph?
We call a connected Salem graph $G$ an \emph{$m$-Salem graph} if $m$ is minimal such that there exists a set of $m$ vertices $v_1, \ldots, v_m$ for which the induced graph $G\setminus \left\{v_1, \ldots, v_m\right\}$ is cyclotomic.
The $1$-Salem graphs are both spectrally and combinatorially close to being cyclotomic: there is a single eigenvalue greater than $2$, and there exists some vertex whose deletion leaves a cyclotomic induced subgraph.
Our main theorem is a complete classification of this especially interesting subset of the Salem graphs.
We remark that any Salem graph contains a $1$-Salem graph as an induced subgraph, so that our classification reveals something of the necessary substructure of an arbitrary Salem graph.
It is worth observing that our definition requires an $m$-Salem graph to be connected, but that when considering induced cyclotomic subgraphs these subgraphs need not be connected.

The description of bipartite $1$-Salem graphs is rather trivial, and is essentially contained in \cite{MS2005}. For non-bipartite Salem graphs all eigenvalues are at least $-2$, and Cameron \textit{et al} (\cite{CGSS1976}) showed that the connected graphs with this property are either generalized line graphs or belong to a finite list of graphs represented in the root system $E_8$. The latter are disposed of by a finite computation.  The main work in our classification comes from a description of the $1$-Salem graphs that are generalised line graphs. We summarise our 
results in the following Theorem.

\begin{theorem} \label{T:summary}
\begin{enumerate}[(i)]
\item All bipartite $1$-Salem graphs are described by Theorem \ref{T:bipartite} (\S\ref{S:bipartite}).
\item Every $1$-Salem graph that is a generalised line graph is described by Theorem \ref{T:glg_result}; there are $25$ infinite families and $6$ sporadic examples (\S\ref{S:glgs}).
\item There are $377$ non-bipartite $1$-Salem graphs that are not generalised line graphs: see Theorem \ref{T:e8_result} (\S\ref{S:E8}).
\end{enumerate}
\end{theorem}

In Section \ref{S:background} we state some definitions and results that we will need to prove the three components of Theorem \ref{T:summary}.
The proof of Theorem \ref{T:summary} is split over three sections: Section \ref{S:bipartite} for the bipartite graphs; Section \ref{S:glgs} for the generalised line graphs; and Section \ref{S:E8} for the non-bipartite graphs that are not generalised line graphs.

\section{Background}\label{S:background}

For convenience we list some standard definitions and results that we shall exploit in the proof of Theorem \ref{T:summary}.
We refer to the first result as `interlacing'.  It allows us to bound the second-largest eigenvalue of a graph by the index of any induced subgraph obtained by deleting a single vertex.

\begin{theorem}[Cauchy \cite{Ca1829}; or see \cite{GR2001}, Theorem 9.1.1] \label{T:interlacing}
Let $G$ be an $n$-vertex graph with vertex set $V\left(G\right)$ and eigenvalues $\lambda_1 \geq \lambda_2 \geq \ldots \geq \lambda_n$. Also, let $H$ be the induced graph on $V\left(G\right)\backslash\left\{v\right\}$ obtained from $G$ by deleting the vertex $v$ and its incident edges.
Then the eigenvalues $\mu_1 \geq \mu_2 \geq \ldots \geq \mu_{n-1}$ of $H$ interlace with those of $G$; that is
\begin{displaymath}
\lambda_1 \geq \mu_1 \geq \lambda_2 \geq \mu_2 \geq \ldots \geq \mu_{n-1} \geq \lambda_n.
\end{displaymath}
\end{theorem}

We also need the following definitions from graph theory: see \cite{CRS2004} or \cite{GR2001}. Let $G$ be an $n$-vertex graph or multigraph. The \emph{line graph} $L\left(G\right)$ is the graph whose vertices are the edges of $G$, with two vertices in $L\left(G\right)$ adjacent whenever the corresponding edges in $G$ have exactly one vertex in common. The \emph{root graph} of a line (multi)graph $L\left(G\right)$ is simply $G$ itself (with a small number of exceptions, $G$ is uniquely determined by $L(G)$). Now let $a_1, \ldots ,a_n$ be non-negative integers. The \emph{generalized line graph} $L\left(G;a_1, \ldots ,a_n\right)$ is the graph $L(\hat{G})$, where $\hat{G}$ is the multigraph $G\left(a_1, \ldots ,a_n\right)$ obtained from $G$ by adding $a_i$ pendant 2-cycles at vertex $v_i$ $\left(i=1,\ldots,n\right)$. An \emph{internal path} of a graph $G$ is a sequence of vertices $v_1, \ldots, v_k$ of $G$ such that all vertices are distinct (except possibly $v_1$ and $v_k$), $v_i$ is adjacent to $v_{i+1}$ (for $i=1,\ldots,k-1$), $v_1$ and $v_k$ have degree at least $3$, and all of $v_2,\ldots,v_{k-1}$ have degree $2$.

We define a \emph{generalized cocktail party graph (GCP)} as a graph isomorphic to a complete graph (or clique) but with some independent edges removed; that is a graph where all vertices have degree $n-1$ or $n-2$. We will use $GCP\left(n,m\right)$ to denote the GCP on $n$ vertices with $m$ edges removed, where $0\leq m\leq\lfloor n/2\rfloor$. In $GCP\left(n,m\right)$, we will refer to a vertex of degree $n-1$ as being `of maximal degree'.

\begin{theorem}[see \cite{CRS2004}, Theorem 2.3.1 and Theorem 2.1.1] \label{T:glgcond} 
A connected graph is a generalized line graph if and only if its edges can be partitioned into GCPs such that
\begin{enumerate}[(i)]
	\item two GCPs have at most one common vertex;
	\item each vertex is no more than two GCPs;
	\item if two GCPs have a common vertex, then it is of maximal degree in both of them.
\end{enumerate}
Also, a graph is a line graph if and only if its edges can be partitioned in such a way that every edge is in a clique and no vertex is in more than two cliques.
\end{theorem}

A further important tool in the proof of Theorem \ref{T:glg_result} is the following:

\begin{proposition}[\cite{MS2005}, Proposition 3.2] \label{L:hamlemma}
Let $G$ be a connected graph with $\lambda_1>2$ and $\lambda_2\leq2$, then the vertices $V$ of $G$ can be partitioned as $V = M \cup A \cup H$ where
\begin{enumerate}[(i)]
	\item the induced subgraph $G\vert_M$ is minimal subject to $\lambda_1(G\vert_M)>2$;
	\item the set $A$ consists of all vertices of $G\vert_{A \cup H}$ adjacent to some vertex of $M$;
	\item the induced subgraph $G\vert_H$ is cyclotomic.
\end{enumerate}
\end{proposition}

\section{All bipartite $1$-Salem graphs}\label{S:bipartite}
We now begin our proof of Theorem \ref{T:summary}, starting with part (i); the bipartite $1$-Salem graphs.
\begin{theorem}\label{T:bipartite}
Let $H_1,\dots,H_s$ be a finite set of connected cyclotomic bipartite graphs (so excluding odd cycles from the connected cyclotomic graphs of Figure \ref{F:cyclograph}), and for each $i$ ($1\le i\le s$) let $S_i$ be a non-empty subset of the vertices of $H_i$ such that all the vertices of $S_i$ fall in the same subset of the bipartition of $H_i$.
Form the graph $G$ by taking the union of all the $H_i$ along with a new vertex $v$, and with edges joining $v$ to each vertex in each $S_i$. Then
\begin{enumerate}[(i)]
\item unless $G$ is cyclotomic, it is $1$-Salem;
\item all connected bipartite $1$-Salem graphs arise in this way.
\end{enumerate}
\end{theorem}

\begin{proof}
The second part is clear: if $G$ is a bipartite $1$-Salem graph, then there exists a vertex $v$ whose deletion leaves a cyclotomic induced subgraph, and the components of this give the $H_i$.
The first part is a consequence of interlacing: $G$ has at most one eigenvalue greater than $2$, and being bipartite we are done.
\end{proof}

One can be more explicit and describe precisely which combinations of $H_i$ and $S_i$ result in $G$ being cyclotomic, simply by running through the possibilities for cyclotomic $G$ given in Figure \ref{F:cyclograph}.  Table \ref{tab:bpcycs2} below shows explicitly which cyclotomic graphs can arise for each choice of $H_i$ ($1\le i\le s$). The subsets $S_i$ that give the graphs $G$ are easy to spot. As the degree of $v$ is at least $s$, we see that no cyclotomic graphs arise when $s\geq 5$.

\begin{table}[h]
\[
\begin{array}{|c|l|c||c|l|c|}\hline
s & H_1,\ldots,H_s & G                             & s & H_1,\ldots,H_s & G\\ \hline
\rule{0pt}{12pt} 1 & H_1=E_6 & \tilde{E}_6                          & 2 & H_1=K_1,\, H_2=P_5 & \tilde{E}_6 \\
  & H_1=P_7 & \tilde{E}_7                          &   & H_1=K_1,\, H_2=D_6 & \tilde{E}_7 \\
  & H_1=E_7 & \tilde{E}_7                          &   & H_1=K_2,\, H_2=P_5 & \tilde{E}_7 \\
  & H_1=P_8 & \tilde{E}_8                          &   & H_1=K_1,\, H_2=P_7 & \tilde{E}_8 \\
  & H_1=D_8 & \tilde{E}_8                          &   & H_1=H_2=P_4 & \tilde{E}_8 \\
  & H_1=E_8 & \tilde{E}_8                          &   & H_1=P_3,\, H_2=D_5 & \tilde{E}_8 \\
  & H_1=P_n & \tilde{A}_n                          &   & H_1=P_2,\, H_2=E_6 & \tilde{E}_8 \\
  & H_1=D_n & \tilde{D}_n                          &   & H_1=K_1,\, H_2=E_7 & \tilde{E}_8 \\
  &         &                                      &   & H_1=D_{n_1},\, H_2=D_{n_2} & \tilde{D}_n \\ \hline
\rule{0pt}{12pt} 3 & H_1=H_2=H_3=K_2 & \tilde{E}_6                  & 4 & H_1=H_2=H_3=H_4=K_1 & \tilde{D}_4 \\
  & H_1=K_1,\, H_2=H_3=P_3 & \tilde{E}_7      &&& \\
  & H_1=K_1,\, H_2=K_2,\, H_3=P_5 & \tilde{E}_8 &&& \\
  & H_1=H_2=K_2,\, H_3=D_{n-2} & \tilde{D}_n  &&& \\ \hline
\end{array}
\]
\caption{The cyclotomic graphs that may arise in Theorem \ref{T:bipartite}, arranged by the number of components $s$.  See \cite{CRS2004} for the (standard) definitions of $E_6$, $E_7$, $E_8$, $A_n$, $D_n$; $P_n$ is the path on $n$ vertices.}
\label{tab:bpcycs2}
\end{table}

\section{All $1$-Salem generalised line graphs}\label{S:glgs}
\begin{theorem}\label{T:glg_result}
The $1$-Salem generalized line graphs are precisely the graphs in Figures \ref{F:1salemglg1}, \ref{F:1salemglg2} and \ref{F:1salemglg3}.
\end{theorem}

\begin{figure}[h]
\begin{center}
\begin{tabular}{cccc}
\begin{tikzpicture}[scale=0.42, auto] 
\foreach \pos/\name in
{{(0,0)/a},{(0.866,0.5)/b},{(0.866,-0.5)/c},{(-2,0)/aa},{(1.866,2.232)/ba},{(1.866,-2.232)/ca}}
\node[vertex] (\name) at \pos {};
\foreach \pos/\name in
{{(-2.866,0.5)/ab},{(-2.866,-0.5)/ac},{(1.866,3.232)/bb},{(2.732,2.732)/bc},{(2.732,-2.732)/cb},{(1.866,-3.232)/cc}}
\node[vertexg] (\name) at \pos {};
\foreach \edgetype/\source/ \dest in {pedge/a/b,pedge/b/c,pedge/c/a,gedge/aa/ab,gedge/aa/ac,gedge/ba/bb,gedge/ba/bc,gedge/ca/cb,gedge/ca/cc,nedge/a/aa,nedge/b/ba,nedge/c/ca}
\path[\edgetype] (\source) -- (\dest);
\foreach \pos/\name in
{{(0,-3.232)/za},{(0,3.232)/zb}}
\node[] (\name) at \pos {};
\node[] at (-1,0.5) {$a$};
\node[] at (0.95,1.6) {$b$};
\node[] at (0.95,-1.6) {$c$};
\end{tikzpicture}
&
\begin{tikzpicture}[scale=0.42, auto] 
\foreach \pos/\name in
{{(0,0)/a},{(0.866,0.5)/b},{(0.866,-0.5)/c},{(-2,0)/aa}}
\node[vertex] (\name) at \pos {};
\foreach \pos/\name in
{{(-2.866,0.5)/ab},{(-2.866,-0.5)/ac}}
\node[vertexg] (\name) at \pos {};
\foreach \edgetype/\source/ \dest in {pedge/a/b,pedge/b/c,pedge/c/a,gedge/aa/ab,gedge/aa/ac,nedge/a/aa}
\path[\edgetype] (\source) -- (\dest);
\draw[densely dashed] (0.866,-0.5) arc (-160.529:160.529:1.5cm);
\draw (0.866,0.5) arc (160.529:53.51:1.5cm);
\draw (0.866,-0.5) arc (-160.529:-53.51:1.5cm);
\foreach \pos/\name in
{{(0,-3.232)/za},{(0,3.232)/zb}}
\node[] (\name) at \pos {};
\node[] at (-1,0.5) {$a$};
\node[] at (3.2,0) {$b$};
\end{tikzpicture}
&
\begin{tikzpicture}[scale=0.42, auto] 
\foreach \pos/\name in
{{(-0.866,0.5)/a},{(0.866,0.5)/b},{(-0.866,-0.5)/c},{(0.866,-0.5)/d},{(0,0)/e},{(-1.866,2.232)/aa},{(1.866,2.232)/ba},{(-1.866,-2.232)/ca},{(1.866,-2.232)/da}}
\node[vertex] (\name) at \pos {};
\foreach \pos/\name in
{{(-1.866,3.232)/ab},{(-2.732,2.732)/ac},{(1.866,3.232)/bb},{(2.732,2.732)/bc},{(-2.732,-2.732)/cb},{(-1.866,-3.232)/cc},{(2.732,-2.732)/db},{(1.866,-3.232)/dc}}
\node[vertexg] (\name) at \pos {};
\foreach \edgetype/\source/ \dest in {pedge/e/a,pedge/e/b,pedge/e/c,pedge/e/d,pedge/a/c,pedge/b/d,gedge/aa/ab,gedge/aa/ac,gedge/ba/bb,gedge/ba/bc,gedge/ca/cb,gedge/ca/cc,gedge/da/db,gedge/da/dc,nedge/a/aa,nedge/b/ba,nedge/c/ca,nedge/d/da}
\path[\edgetype] (\source) -- (\dest);
\foreach \pos/\name in
{{(0,-3.232)/za},{(0,3.232)/zb}}
\node[] (\name) at \pos {};
\node[] at (-0.95,1.6) {$a$};
\node[] at (0.95,1.6) {$b$};
\node[] at (-0.95,-1.6) {$c$};
\node[] at (0.95,-1.6) {$d$};
\end{tikzpicture}
&
\begin{tikzpicture}[scale=0.42, auto] 
\foreach \pos/\name in
{{(-0.866,0.5)/a},{(0.866,0.5)/b},{(-0.866,-0.5)/c},{(0.866,-0.5)/d},{(0,0)/e},{(-1.866,2.232)/aa},{(1.866,2.232)/ba}}
\node[vertex] (\name) at \pos {};
\foreach \pos/\name in
{{(-1.866,3.232)/ab},{(-2.732,2.732)/ac},{(1.866,3.232)/bb},{(2.732,2.732)/bc}}
\node[vertexg] (\name) at \pos {};
\foreach \edgetype/\source/ \dest in {pedge/e/a,pedge/e/b,pedge/e/c,pedge/e/d,pedge/a/c,pedge/b/d,gedge/aa/ab,gedge/aa/ac,gedge/ba/bb,gedge/ba/bc,nedge/a/aa,nedge/b/ba}
\path[\edgetype] (\source) -- (\dest);
\draw[densely dashed] (0.866,-0.5) arc (54.737:-234.737:1.5cm);
\draw (0.866,-0.5) arc (54.737:-41.754:1.5cm);
\draw (-0.866,-0.5) arc (-234.737:-138.245:1.5cm);
\foreach \pos/\name in
{{(0,-3.232)/za},{(0,3.232)/zb}}
\node[] (\name) at \pos {};
\node[] at (-0.95,1.6) {$a$};
\node[] at (0.95,1.6) {$b$};
\node[] at (0,-2.75) {$c$};
\end{tikzpicture}

\\
$G_1\left(a,b,c\right)$, $a\geq1$ & $G_2\left(a,b\right)$, $b\geq2$ & $G_3\left(a,b,c,d\right)$ & $G_4\left(a,b,c\right), c\ge1$ \\
or $a=\hat{0}$ & & &
\end{tabular}
\end{center}

\begin{center}
\begin{tabular}{cccc}
\begin{tikzpicture}[scale=0.42, auto] 
\foreach \pos/\name in
{{(-0.866,0.5)/a},{(0.866,0.5)/b},{(-0.866,-0.5)/c},{(0.866,-0.5)/d},{(0,0)/e}}
\node[vertex] (\name) at \pos {};
\foreach \edgetype/\source/ \dest in {pedge/e/a,pedge/e/b,pedge/e/c,pedge/e/d,pedge/a/c,pedge/b/d}
\path[\edgetype] (\source) -- (\dest);
\draw[densely dashed] (0.866,-0.5) arc (54.737:-234.737:1.5cm);
\draw (0.866,-0.5) arc (54.737:-41.754:1.5cm);
\draw (-0.866,-0.5) arc (-234.737:-138.245:1.5cm);
\draw[densely dashed] (0.866,0.5) arc (-54.737:234.737:1.5cm);
\draw (0.866,0.5) arc (-54.737:41.754:1.5cm);
\draw (-0.866,0.5) arc (-125.263:-221.75433:1.5cm);
\foreach \pos/\name in
{{(0,-3.232)/za},{(0,3.232)/zb}}
\node[] (\name) at \pos {};
\node[] at (0,2.75) {$a$};
\node[] at (0,-2.75) {$b$};
\end{tikzpicture}
&
\begin{tikzpicture}[scale=0.42, auto] 
\foreach \pos/\name in
{{(-0.866,0.5)/a},{(0.866,0.5)/b},{(-0.866,-0.5)/c},{(0.866,-0.5)/d},{(0,0)/e},{(-1.866,2.232)/aa},{(-1.866,-2.232)/ca}}
\node[vertex] (\name) at \pos {};
\foreach \pos/\name in
{{(-1.866,3.232)/ab},{(-2.732,2.732)/ac},{(-2.732,-2.732)/cb},{(-1.866,-3.232)/cc}}
\node[vertexg] (\name) at \pos {};
\foreach \edgetype/\source/ \dest in {pedge/e/a,pedge/e/b,pedge/e/c,pedge/e/d,pedge/a/c,pedge/b/d,gedge/aa/ab,gedge/aa/ac,gedge/ca/cb,gedge/ca/cc,nedge/a/aa,nedge/c/ca}
\path[\edgetype] (\source) -- (\dest);
\draw[densely dashed] (0.866,-0.5) arc (-160.529:160.529:1.5cm);
\draw (0.866,0.5) arc (160.529:53.51:1.5cm);
\draw (0.866,-0.5) arc (-160.529:-53.51:1.5cm);
\foreach \pos/\name in
{{(0,-3.232)/za},{(0,3.232)/zb}}
\node[] (\name) at \pos {};
\node[] at (-0.95,1.6) {$a$};
\node[] at (3.2,0) {$b$};
\node[] at (-0.95,-1.6) {$c$};
\end{tikzpicture}
&
\begin{tikzpicture}[scale=0.42, auto] 
\foreach \pos/\name in
{{(-0.866,0.5)/a},{(0.866,0.5)/b},{(-0.866,-0.5)/c},{(0.866,-0.5)/d},{(0,0)/e}}
\node[vertex] (\name) at \pos {};
\foreach \edgetype/\source/ \dest in {pedge/e/a,pedge/e/b,pedge/e/c,pedge/e/d,pedge/a/c,pedge/b/d}
\path[\edgetype] (\source) -- (\dest);
\draw[densely dashed] (0.866,-0.5) arc (-160.529:160.529:1.5cm);
\draw (0.866,0.5) arc (160.529:53.51:1.5cm);
\draw (0.866,-0.5) arc (-160.529:-53.51:1.5cm);
\draw[densely dashed] (-0.866,0.5) arc (19.471:340.529:1.5cm);
\draw (-0.866,0.5) arc (19.471:126.49:1.5cm);
\draw (-0.866,-0.5) arc (-19.471:-126.49:1.5cm);
\foreach \pos/\name in
{{(0,-3.8)/za},{(0,1.866)/zb}}
\node[] (\name) at \pos {};
\node[] at (-3.2,0) {$a$};
\node[] at (3.2,0) {$b$};
\end{tikzpicture}
&
\begin{tikzpicture}[scale=0.42, auto] 
\foreach \pos/\name in
{{(0,0)/a},{(0,1)/b},{(-0.5,1.866)/c},{(0.5,1.866)/d},{(-0.5,-0.866)/e},{(0.5,-0.866)/f},{(-2.232,-1.866)/g},{(2.232,-1.866)/h}}
\node[vertex] (\name) at \pos {};
\foreach \pos/\name in
{{(-3.232,-1.866)/ga},{(-2.732,-2.732)/gb},{(3.232,-1.866)/ha},{(2.732,-2.732)/hb}}
\node[vertexg] (\name) at \pos {};
\foreach \edgetype/\source/ \dest in {pedge/a/b,pedge/b/c,pedge/b/d,pedge/d/c,pedge/a/e,pedge/e/f,pedge/f/a,nedge/e/g,nedge/f/h,gedge/g/ga,gedge/g/gb,gedge/h/ha,gedge/h/hb}
\path[\edgetype] (\source) -- (\dest);
\foreach \pos/\name in
{{(0,-3.8)/za},{(0,1.866)/zb}}
\node[] (\name) at \pos {};
\node[] at (-1.6,-0.8) {$a$};
\node[] at (1.6,-0.8) {$b$};
\end{tikzpicture}
\\
 $G_5\left(a,b\right),\ a,b\ge1$ & $G_6\left(a,b,c\right)$, $b\geq2$ & $G_7\left(a,b\right)$, $a,b\geq2$  & $G_8\left(a,b\right)$
\end{tabular}
\end{center}

\begin{center}
\begin{tabular}{cccc}
\begin{tikzpicture}[scale=0.42, auto] 
\foreach \pos/\name in
{{(0,0)/a},{(0,1)/b},{(-0.5,1.866)/c},{(0.5,1.866)/d},{(-0.5,-0.866)/e},{(0.5,-0.866)/f}}
\node[vertex] (\name) at \pos {};
\foreach \edgetype/\source/ \dest in {pedge/a/b,pedge/b/c,pedge/b/d,pedge/d/c,pedge/a/e,pedge/e/f,pedge/f/a}
\path[\edgetype] (\source) -- (\dest);
\draw[densely dashed] (0.5,-0.866) arc (70.529:-250.529:1.5cm);
\draw (0.5,-0.866) arc (70.529:-41.754:1.5cm);
\draw (-0.5,-0.866) arc (-250.259:-138.245:1.5cm);
\foreach \pos/\name in
{{(0,-4.098)/za},{(0,2.366)/zb}}
\node[] (\name) at \pos {};
\node[] at (0,-3.2) {$a$};
\end{tikzpicture}
&
\begin{tikzpicture}[scale=0.42, auto] 
\foreach \pos/\name in
{{(0,0)/a},{(1.414,0)/b},{(0.707,-0.707)/c},{(0.707,0.707)/d},{(-2,0)/e},{(3.414,0)/h}}
\node[vertex] (\name) at \pos {};
\foreach \pos/\name in
{{(4.28,0.5)/ha},{(4.28,-0.5)/hb},{(-2.866,0.5)/ea},{(-2.866,-0.5)/eb}}
\node[vertexg] (\name) at \pos {};
\foreach \edgetype/\source/ \dest in {pedge/a/b,pedge/a/c,nedge/a/e,pedge/a/d,pedge/b/c,pedge/b/d,gedge/e/ea,gedge/eb/e,nedge/b/h,gedge/h/ha,gedge/h/hb}
\path[\edgetype] (\source) -- (\dest);
\foreach \pos/\name in
{{(0,-3.232)/za},{(0,3.232)/zb}}
\node[] (\name) at \pos {};
\node[] at (-1,0.5) {$a$};
\node[] at (2.414,0.5) {$b$};
\end{tikzpicture}
&
\begin{tikzpicture}[scale=0.42, auto] 
\foreach \pos/\name in
{{(0,0)/a},{(1.414,0)/b},{(0.707,-0.707)/c},{(0.707,0.707)/d}}
\node[vertex] (\name) at \pos {};
\foreach \edgetype/\source/ \dest in {pedge/a/b,pedge/a/c,pedge/a/d,pedge/b/c,pedge/b/d}
\path[\edgetype] (\source) -- (\dest);
\draw[densely dashed] (1.414,0) arc (61.879:-241.88:1.5cm);
\draw (1.414,0) arc (61.879:-41.754:1.5cm);
\draw (0,0) arc (-241.88:-138.245:1.5cm);
\foreach \pos/\name in
{{(0,-3.232)/za},{(0,3.232)/zb}}
\node[] (\name) at \pos {};
\node[] at (0.707,-2.4) {$a$};
\end{tikzpicture}
&
\begin{tikzpicture}[scale=0.42, auto] 
\foreach \pos/\name in
{{(0,0)/a},{(0.866,0.5)/b},{(0.866,-0.5)/c},{(-0.707,0.707)/d},{(-0.707,-0.707)/e},{(-1.414,0)/f},{(-3.414,0)/g},{(1.866,2.232)/ba},{(1.866,-2.232)/ca}}
\node[vertex] (\name) at \pos {};
\foreach \pos/\name in
{{(-4.28,0.5)/gb},{(-4.28,-0.5)/gc},{(1.866,3.232)/bb},{(2.732,2.732)/bc},{(2.732,-2.732)/cb},{(1.866,-3.232)/cc}}
\node[vertexg] (\name) at \pos {};
\foreach \edgetype/\source/ \dest in {pedge/a/b,pedge/b/c,pedge/c/a,gedge/ba/bb,gedge/ba/bc,gedge/ca/cb,gedge/ca/cc,nedge/b/ba,nedge/c/ca,pedge/a/d,pedge/a/e,pedge/a/f,pedge/f/d,pedge/f/e,nedge/f/g,gedge/g/gb,gedge/g/gc}
\path[\edgetype] (\source) -- (\dest);
\foreach \pos/\name in
{{(0,-3.232)/za},{(0,3.232)/zb}}
\node[] (\name) at \pos {};
\node[] at (-2.414,0.5) {$a$};
\node[] at (0.95,1.6) {$b$};
\node[] at (0.95,-1.6) {$c$};
\end{tikzpicture}
\\
$G_9\left(a\right)$, $a\geq2$ & $G_{10}\left(a,b\right)$ & $G_{11}\left(a\right)$, $a\geq2$ & $G_{12}\left(a,b,c\right)$
\end{tabular}
\end{center}
\caption{Twelve of the $25$ infinite families of $1$-Salem generalised line graphs.  The parameters $a,b,c,d$ are $\ge0$ unless specified.}\label{F:1salemglg1}
\end{figure}
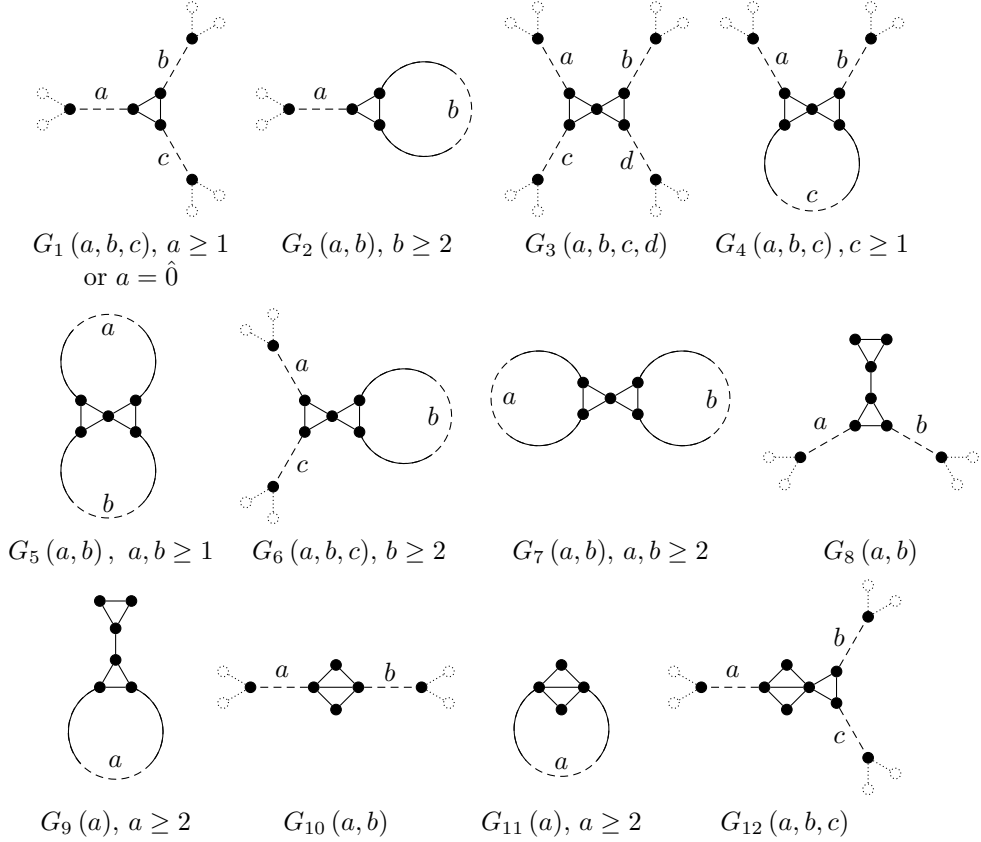

\begin{figure}[h]
\begin{center}
\begin{tabular}{ccc}
\begin{tikzpicture}[scale=0.42, auto] 
\foreach \pos/\name in
{{(0,0)/a},{(0.866,0.5)/b},{(0.866,-0.5)/c},{(-0.707,0.707)/d},{(-0.707,-0.707)/e},{(-1.414,0)/f},{(-3.414,0)/g}}
\node[vertex] (\name) at \pos {};
\foreach \pos/\name in
{{(-4.28,0.5)/ab},{(-4.28,-0.5)/ac}}
\node[vertexg] (\name) at \pos {};
\foreach \edgetype/\source/ \dest in {pedge/a/b,pedge/b/c,pedge/c/a,pedge/a/d,pedge/a/e,pedge/a/f,pedge/f/d,pedge/f/e,nedge/f/g,gedge/g/gb,gedge/g/gc}
\path[\edgetype] (\source) -- (\dest);
\draw[densely dashed] (0.866,-0.5) arc (-160.529:160.529:1.5cm);
\draw (0.866,0.5) arc (160.529:53.51:1.5cm);
\draw (0.866,-0.5) arc (-160.529:-53.51:1.5cm);
\foreach \pos/\name in
{{(0,-2.75)/za},{(0,3.232)/zb}}
\node[] (\name) at \pos {};
\node[] at (-2.414,0.5) {$a$};
\node[] at (3.2,0) {$b$};
\end{tikzpicture}
&
\begin{tikzpicture}[scale=0.42, auto] 
\foreach \pos/\name in
{{(0,0)/a},{(0.866,0.5)/b},{(0.866,-0.5)/c},{(-0.707,0.707)/d},{(-0.707,-0.707)/e},{(-1.414,0)/f},{(1.866,2.232)/ba}}
\node[vertex] (\name) at \pos {};
\foreach \pos/\name in
{{(1.866,3.232)/bb},{(2.732,2.732)/bc}}
\node[vertexg] (\name) at \pos {};
\foreach \edgetype/\source/ \dest in {pedge/a/b,pedge/b/c,pedge/c/a,gedge/ba/bb,gedge/ba/bc,nedge/b/ba,pedge/a/d,pedge/a/e,pedge/a/f,pedge/f/d,pedge/f/e}
\path[\edgetype] (\source) -- (\dest);
\draw[densely dashed] (0.866,-0.5) arc (26.552:-231.292:1.5cm);
\draw (0.866,-0.5) arc (26.552:-59.396:1.5cm);
\draw (-1.414,0) arc (-231.292:-145.346:1.5cm);
\foreach \pos/\name in
{{(0,-2.75)/za},{(0,3.232)/zb}}
\node[] (\name) at \pos {};
\node[] at (-0.6,-2.1) {$a$};
\node[] at (0.95,1.6) {$b$};
\end{tikzpicture}
&
\begin{tikzpicture}[scale=0.42, auto] 
\foreach \pos/\name in
{{(0,0)/a},{(1.414,0)/b},{(0.707,-0.707)/c},{(0.707,0.707)/d},{(-1,0)/e},{(-1.866,0.5)/f},{(-1.866,-0.5)/g},{(3.414,0)/h}}
\node[vertex] (\name) at \pos {};
\foreach \pos/\name in
{{(4.28,0.5)/ha},{(4.28,-0.5)/hb}}
\node[vertexg] (\name) at \pos {};
\foreach \edgetype/\source/ \dest in {pedge/a/b,pedge/a/c,pedge/a/e,pedge/a/d,pedge/b/c,pedge/b/d,pedge/e/g,pedge/f/e,pedge/g/f,nedge/b/h,gedge/h/ha,gedge/h/hb}
\path[\edgetype] (\source) -- (\dest);
\foreach \pos/\name in
{{(0,-2.75)/za},{(0,3.232)/zb}}
\node[] (\name) at \pos {};
\node[] at (2.5,0.5) {$a$};
\end{tikzpicture}

\\
$G_{13}\left(a,b\right)$, $b\geq2$ & $G_{14}\left(a,b\right),\ a\ge1$ & $G_{15}\left(a\right)$
\end{tabular}
\end{center}

\begin{center}
\begin{tabular}{cccc}
\begin{tikzpicture}[scale=0.42, auto] 
\foreach \pos/\name in
{{(0,0)/a},{(1.414,0)/b},{(0.707,-0.707)/c},{(0.707,0.707)/d},{(-0.707,-0.707)/e},{(-0.707,0.707)/f},{(-1.414,0)/g},{(3.414,0)/h},{(-3.414,0)/i}}
\node[vertex] (\name) at \pos {};
\foreach \pos/\name in
{{(4.28,0.5)/ha},{(4.28,-0.5)/hb},{(-4.28,0.5)/ia},{(-4.28,-0.5)/ib}}
\node[vertexg] (\name) at \pos {};
\foreach \edgetype/\source/ \dest in {pedge/a/b,pedge/a/c,pedge/a/e,pedge/a/d,pedge/b/c,pedge/b/d,pedge/f/a,pedge/g/a,pedge/g/f,pedge/g/e,nedge/b/h,gedge/h/ha,gedge/h/hb,nedge/g/i,gedge/i/ib,gedge/i/ia}
\path[\edgetype] (\source) -- (\dest);
\foreach \pos/\name in
{{(0,-3.232)/za},{(0,3.8)/zb}}
\node[] (\name) at \pos {};
\node[] at (-2.414,0.5) {$a$};
\node[] at (2.414,0.5) {$b$};
\end{tikzpicture}
&
\begin{tikzpicture}[scale=0.42, auto] 
\foreach \pos/\name in
{{(0,0)/a},{(1.414,0)/b},{(0.707,-0.707)/c},{(0.707,0.707)/d},{(-0.707,-0.707)/e},{(-0.707,0.707)/f},{(-1.414,0)/g}}
\node[vertex] (\name) at \pos {};
\foreach \edgetype/\source/ \dest in {pedge/a/b,pedge/a/c,pedge/a/e,pedge/a/d,pedge/b/c,pedge/b/d,pedge/f/a,pedge/g/a,pedge/g/f,pedge/g/e}
\path[\edgetype] (\source) -- (\dest);
\draw[densely dashed] (1.414,0) arc (36.1:-216.1:1.75cm);
\draw (1.414,0) arc (36.1:-47.966:1.75cm);
\draw (-1.414,0) arc (-216.1:-132.034:1.75cm);
\foreach \pos/\name in
{{(0,-3.232)/za},{(0,3.232)/zb}}
\node[] (\name) at \pos {};
\node[] at (0,-2.2) {$a$};
\end{tikzpicture}
&
\begin{tikzpicture}[scale=0.42, auto] 
\foreach \pos/\name in
{{(0,0)/a},{(-1.414,0)/b},{(-0.707,0.707)/c},{(-0.707,-0.707)/d},{(2,0)/e}}
\node[vertex] (\name) at \pos {};
\foreach \pos/\name in
{{(2.866,0.5)/eb},{(2.866,-0.5)/ec}}
\node[vertexg] (\name) at \pos {};
\foreach \edgetype/\source/ \dest in {pedge/a/b,pedge/a/c,pedge/a/d,pedge/b/c,pedge/b/d,pedge/c/d,nedge/a/e,gedge/e/eb,gedge/e/ec}
\path[\edgetype] (\source) -- (\dest);
\foreach \pos/\name in
{{(0,-3.232)/za},{(0,3.232)/zb}}
\node[] (\name) at \pos {};
\node[] at (1,0.5) {$a$};
\end{tikzpicture}
&
\begin{tikzpicture}[scale=0.42, auto] 
\foreach \pos/\name in
{{(0,0)/a},{(0.866,0.5)/b},{(0.866,-0.5)/c},{(-1.414,0)/d},{(-0.707,0.707)/e},{(-0.707,-0.707)/f},{(1.866,2.232)/ba},{(1.866,-2.232)/ca}}
\node[vertex] (\name) at \pos {};
\foreach \pos/\name in
{{(1.866,3.232)/bb},{(2.732,2.732)/bc},{(2.732,-2.732)/cb},{(1.866,-3.232)/cc}}
\node[vertexg] (\name) at \pos {};
\foreach \edgetype/\source/ \dest in {pedge/a/b,pedge/b/c,pedge/c/a,pedge/a/f,pedge/a/d,pedge/a/e,pedge/f/g,pedge/d/e,pedge/f/e,gedge/ba/bb,gedge/ba/bc,gedge/ca/cb,gedge/ca/cc,nedge/b/ba,nedge/c/ca}
\path[\edgetype] (\source) -- (\dest);
\foreach \pos/\name in
{{(0,-3.232)/za},{(0,3.232)/zb}}
\node[] (\name) at \pos {};
\node[] at (0.95,1.6) {$a$};
\node[] at (0.95,-1.6) {$b$};
\end{tikzpicture}
\\
$G_{16}\left(a,b\right)$ & $G_{17}\left(a\right),\ a\ge1$ & $G_{18}\left(a\right)$ & $G_{19}\left(a,b\right)$
\end{tabular}
\end{center}

\begin{center}
\begin{tabular}{cccc}

\begin{tikzpicture}[scale=0.42, auto] 
\foreach \pos/\name in
{{(0,0)/a},{(0.866,0.5)/b},{(0.866,-0.5)/c},{(-1.414,0)/d},{(-0.707,0.707)/e},{(-0.707,-0.707)/f}}
\node[vertex] (\name) at \pos {};
\foreach \edgetype/\source/ \dest in {pedge/a/b,pedge/b/c,pedge/c/a,pedge/a/f,pedge/a/d,pedge/a/e,pedge/f/g,pedge/d/e,pedge/f/e}
\path[\edgetype] (\source) -- (\dest);
\draw[densely dashed] (0.866,-0.5) arc (-160.529:160.529:1.5cm);
\draw (0.866,0.5) arc (160.529:53.51:1.5cm);
\draw (0.866,-0.5) arc (-160.529:-53.51:1.5cm);
\foreach \pos/\name in
{{(0,-3.232)/za},{(0,3.232)/zb}}
\node[] (\name) at \pos {};
\node[] at (3.2,0) {$a$};
\end{tikzpicture}
&
\begin{tikzpicture}[scale=0.42, auto] 
\foreach \pos/\name in
{{(0,0)/a},{(1.414,0)/b},{(0.707,-0.707)/c},{(0.707,0.707)/d},{(-0.707,-0.707)/e},{(-0.707,0.707)/f},{(-1.414,0)/g},{(3.414,0)/h}}
\node[vertex] (\name) at \pos {};
\foreach \pos/\name in
{{(4.28,0.5)/ha},{(4.28,-0.5)/hb}}
\node[vertexg] (\name) at \pos {};
\foreach \edgetype/\source/ \dest in {pedge/a/b,pedge/a/c,pedge/a/e,pedge/a/d,pedge/b/c,pedge/b/d,pedge/f/a,pedge/g/a,pedge/g/f,pedge/g/e,pedge/f/e,nedge/b/h,gedge/h/ha,gedge/h/hb}
\path[\edgetype] (\source) -- (\dest);
\foreach \pos/\name in
{{(0,-3.232)/za},{(0,3.232)/zb}}
\node[] (\name) at \pos {};
\node[] at (2.5,0.5) {$a$};
\end{tikzpicture}
&
\begin{tikzpicture}[scale=0.42, auto] 
\foreach \pos/\name in
{{(0,0)/a},{(-0.588,0.809)/b},{(-1.54,0.5)/c},{(-1.54,-0.5)/d},{(-0.588,-0.809)/e},{(2,0)/f}}
\node[vertex] (\name) at \pos {};
\foreach \pos/\name in
{{(2.866,0.5)/fa},{(2.866,-0.5)/fb}}
\node[vertexg] (\name) at \pos {};
\foreach \edgetype/\source/ \dest in {pedge/a/b,pedge/a/c,pedge/a/d,pedge/a/e,pedge/b/e,pedge/b/c,pedge/c/d,pedge/d/e,nedge/a/f,gedge/f/fa,gedge/f/fb}
\path[\edgetype] (\source) -- (\dest);
\foreach \pos/\name in
{{(0,-3.232)/za},{(0,3.232)/zb}}
\node[] (\name) at \pos {};
\node[] at (1,0.5) {$a$};
\end{tikzpicture}
&
\begin{tikzpicture}[scale=0.42, auto] 
\foreach \pos/\name in
{{(0,0)/a},{(-0.588,0.809)/b},{(-1.54,0.5)/c},{(-1.54,-0.5)/d},{(-0.588,-0.809)/e},{(.866,0.5)/f},{(.866,-0.5)/g},{(1.866,2.232)/fa},{(1.866,-2.232)/ga}}
\node[vertex] (\name) at \pos {};
\foreach \pos/\name in
{{(1.866,3.232)/fb},{(2.732,2.732)/fc},{(2.732,-2.732)/gb},{(1.866,-3.232)/gc}}
\node[vertexg] (\name) at \pos {};
\foreach \edgetype/\source/ \dest in {pedge/a/b,pedge/a/c,pedge/a/d,pedge/a/e,pedge/b/e,pedge/b/c,pedge/c/d,pedge/d/e,pedge/a/f,pedge/a/g,pedge/g/f,nedge/f/fa,nedge/g/ga,gedge/fa/fb,gedge/fa/fc,gedge/ga/gb,gedge/ga/gc}
\path[\edgetype] (\source) -- (\dest);
\foreach \pos/\name in
{{(0,-3.232)/za},{(0,3.232)/zb}}
\node[] (\name) at \pos {};
\node[] at (0.95,1.6) {$a$};
\node[] at (0.95,-1.6) {$b$};
\end{tikzpicture}
\\
$G_{20}\left(a\right)$, $a\geq2$ & $G_{21}\left(a\right)$ & $G_{22}\left(a\right)$ & $G_{23}\left(a,b\right)$
\end{tabular}
\end{center}


\begin{center}
\begin{tabular}{cc}
\begin{tikzpicture}[scale=0.42, auto] 
\foreach \pos/\name in
{{(0,0)/a},{(-0.588,0.809)/b},{(-1.54,0.5)/c},{(-1.54,-0.5)/d},{(-0.588,-0.809)/e},{(.866,0.5)/f},{(.866,-0.5)/g}}
\node[vertex] (\name) at \pos {};
\foreach \edgetype/\source/ \dest in {pedge/a/b,pedge/a/c,pedge/a/d,pedge/a/e,pedge/b/e,pedge/b/c,pedge/c/d,pedge/d/e,pedge/a/f,pedge/a/g,pedge/g/f}
\path[\edgetype] (\source) -- (\dest);
\draw[densely dashed] (0.866,-0.5) arc (-160.529:160.529:1.5cm);
\draw (0.866,0.5) arc (160.529:53.51:1.5cm);
\draw (0.866,-0.5) arc (-160.529:-53.51:1.5cm);
\foreach \pos/\name in
{{(0,-1.5)/za},{(0,1.5)/zb}}
\node[] (\name) at \pos {};
\node[] at (3.2,0) {$a$};
\end{tikzpicture}
&
\begin{tikzpicture}[scale=0.42, auto] 
\foreach \pos/\name in
{{(0,0)/a},{(-0.588,0.809)/b},{(-1.54,0.5)/c},{(-1.54,-0.5)/d},{(-0.588,-0.809)/e},{(0.707,0.707)/f},{(0.707,-0.707)/g},{(1.414,0)/h},{(3.414,0)/ha}}
\node[vertex] (\name) at \pos {};
\foreach \pos/\name in
{{(4.28,0.5)/hb},{(4.28,-0.5)/hc}}
\node[vertexg] (\name) at \pos {};
\foreach \edgetype/\source/ \dest in {pedge/a/b,pedge/a/c,pedge/a/d,pedge/a/e,pedge/b/e,pedge/b/c,pedge/c/d,pedge/d/e,pedge/a/f,pedge/a/g,pedge/g/h,pedge/f/h,pedge/a/h,nedge/h/ha,gedge/ha/hb,gedge/ha/hc}
\path[\edgetype] (\source) -- (\dest);
\foreach \pos/\name in
{{(0,-1.5)/za},{(0,1.5)/zb}}
\node[] (\name) at \pos {};
\node[] at (2.5,0.5) {$a$};
\end{tikzpicture}
\\
$G_{24}\left(a\right)$, $a\geq2$ & $G_{25}\left(a\right)$
\end{tabular}
\end{center}
\caption{Thirteen of the $25$ infinite families of $1$-Salem generalised line graphs.  The parameters $a,b$ are $\ge0$ unless specified.}\label{F:1salemglg2}
\end{figure}
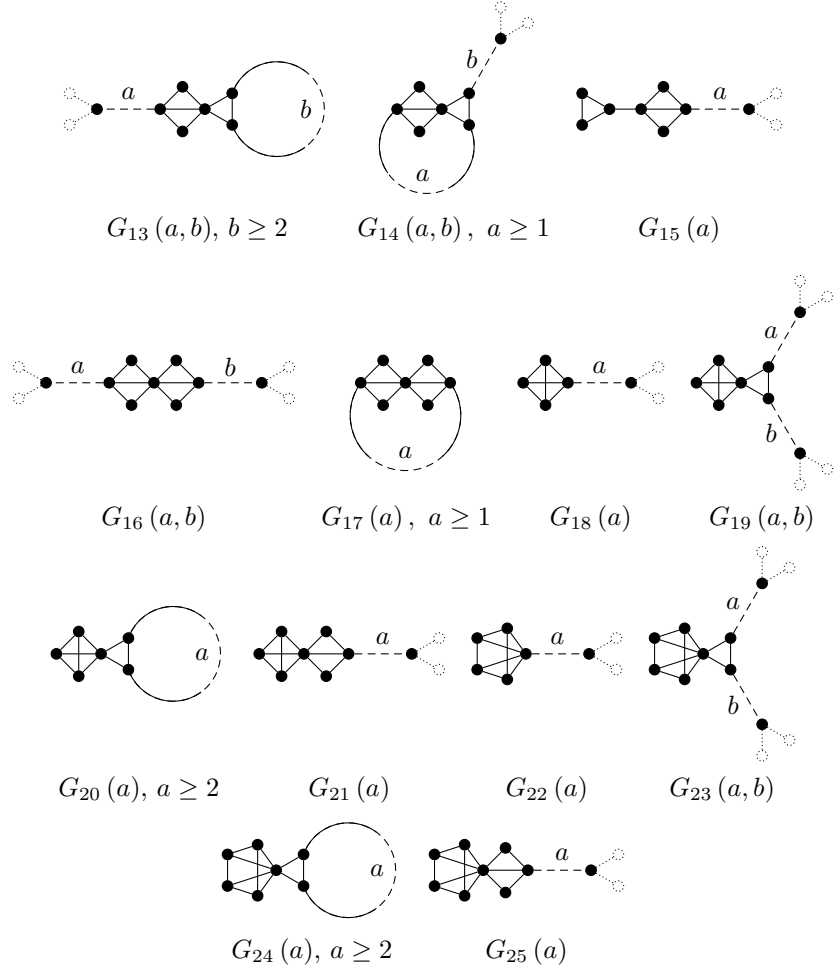

\begin{figure}
\begin{center}
\begin{tabular}{cccccc}
\begin{tikzpicture}[scale=0.42, auto] 
\foreach \pos/\name in
{{(0,0)/a},{(1,0)/b},{(-1,0)/c},{(1.866,0.5)/d},{(-1.866,0.5)/e},{(-1.866,-0.5)/f},{(1.866,-0.5)/g}}
\node[vertex] (\name) at \pos {};
\foreach \edgetype/\source/ \dest in {pedge/a/b,pedge/a/c,pedge/c/e,pedge/c/f,pedge/b/g,pedge/b/d,pedge/d/g,pedge/f/e}
\path[\edgetype] (\source) -- (\dest);
\foreach \pos/\name in
{{(0,-1)/za},{(0,1)/zb}}
\node[] (\name) at \pos {};
\end{tikzpicture}
&
\begin{tikzpicture}[scale=0.42, auto] 
\foreach \pos/\name in
{{(0,0)/a},{(1,0)/b},{(-1.414,0)/c},{(1.866,0.5)/d},{(-0.707,0.707)/e},{(-0.707,-0.707)/f},{(1.866,-0.5)/g}}
\node[vertex] (\name) at \pos {};
\foreach \edgetype/\source/ \dest in {pedge/a/b,pedge/a/c,pedge/a/e,pedge/a/f,pedge/b/g,pedge/b/d,pedge/d/g,pedge/c/e,pedge/c/f,pedge/e/f}
\path[\edgetype] (\source) -- (\dest);
\foreach \pos/\name in
{{(0,-1)/za},{(0,1)/zb}}
\node[] (\name) at \pos {};
\end{tikzpicture}
&
\begin{tikzpicture}[scale=0.42, auto] 
\foreach \pos/\name in
{{(0,0)/a},{(1.414,0)/b},{(-1.414,0)/c},{(0.707,0.707)/d},{(-0.707,0.707)/e},{(-0.707,-0.707)/f},{(0.707,-0.707)/g}}
\node[vertex] (\name) at \pos {};
\foreach \edgetype/\source/ \dest in {pedge/a/b,pedge/a/c,pedge/a/d,pedge/a/e,pedge/a/f,pedge/a/g,pedge/b/g,pedge/b/d,pedge/d/g,pedge/c/e,pedge/c/f,pedge/e/f}
\path[\edgetype] (\source) -- (\dest);
\foreach \pos/\name in
{{(0,-1)/za},{(0,1)/zb}}
\node[] (\name) at \pos {};
\end{tikzpicture}
&
\begin{tikzpicture}[scale=0.42, auto] 
\foreach \pos/\name in
{{(0,0)/a},{(-0.588,0.809)/b},{(-1.54,0.5)/c},{(-1.54,-0.5)/d},{(-0.588,-0.809)/e},{(1,0)/f},{(1.866,0.5)/g},{(1.866,-0.5)/h}}
\node[vertex] (\name) at \pos {};
\foreach \edgetype/\source/ \dest in {pedge/a/b,pedge/a/c,pedge/a/d,pedge/a/e,pedge/b/e,pedge/b/c,pedge/c/d,pedge/d/e,pedge/a/f,pedge/g/f,pedge/h/f,pedge/h/g}
\path[\edgetype] (\source) -- (\dest);
\foreach \pos/\name in
{{(0,-1)/za},{(0,1)/zb}}
\node[] (\name) at \pos {};
\end{tikzpicture}
&
\begin{tikzpicture}[scale=0.42, auto] 
\foreach \pos/\name in
{{(0,0)/a},{(-0.588,0.809)/b},{(-1.54,0.5)/c},{(-1.54,-0.5)/d},{(-0.588,-0.809)/e},{(0.707,0.707)/f},{(0.707,-0.707)/g},{(1.414,0)/h}}
\node[vertex] (\name) at \pos {};
\foreach \edgetype/\source/ \dest in {pedge/a/b,pedge/a/c,pedge/a/d,pedge/a/e,pedge/b/e,pedge/b/c,pedge/c/d,pedge/d/e,pedge/a/f,pedge/a/g,pedge/g/h,pedge/f/h,pedge/a/h,pedge/f/g}
\path[\edgetype] (\source) -- (\dest);
\foreach \pos/\name in
{{(0,-1)/za},{(0,1)/zb}}
\node[] (\name) at \pos {};
\end{tikzpicture}
&
\begin{tikzpicture}[scale=0.42, auto] 
\foreach \pos/\name in
{{(0,0)/a},{(-0.588,0.809)/b},{(-1.54,0.5)/c},{(-1.54,-0.5)/d},{(-0.588,-0.809)/e},{(0.588,0.809)/f},{(1.54,0.5)/g},{(1.54,-0.5)/h},{(0.588,-0.809)/i}}
\node[vertex] (\name) at \pos {};
\foreach \edgetype/\source/ \dest in {pedge/a/b,pedge/a/c,pedge/a/d,pedge/a/e,pedge/b/e,pedge/b/c,pedge/c/d,pedge/d/e,pedge/a/f,pedge/f/g,pedge/g/h,pedge/h/i,pedge/g/a,pedge/h/a,pedge/i/a,pedge/i/f}
\path[\edgetype] (\source) -- (\dest);
\foreach \pos/\name in
{{(0,-1)/za},{(0,1)/zb}}
\node[] (\name) at \pos {};
\end{tikzpicture}\\
$G_{26}$ & $G_{27}$ & $G_{28}$ & $G_{29}$ & $G_{30}$ & $G_{31}$
\end{tabular}
\end{center}
\caption{The six sporadic $1$-Salem generalised line graphs.}
\label{F:1salemglg3}
\end{figure}
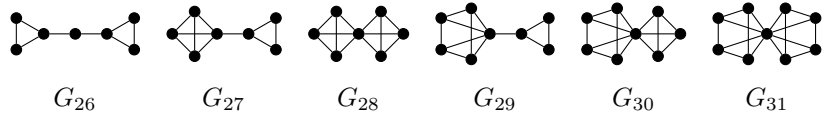

A quick notational point: in Figures \ref{F:1salemglg1}--\ref{F:1salemglg3}, a dashed edge indicates a path between the endpoints of the dashed edge, having an arbitrary number of edges (perhaps even none, or perhaps with a lower bound shown); the parameter attached to a dashed edge gives the number of edges on this path. Dotted edges and vertices are used to indicate edges and vertices that may or may not be there. Here the dotted edges and vertices always form a `snake's tongue' shape, and later on we will use the notation $\hat{a}$ to indicate a path of length $a$ (i.e., having $a$ edges) with two extra vertices in the shape of a snake's tongue on the loose end. Thus, for example, $G_{10}(1,1)$ and $G_{10}(1,\hat{1})$ are shown in Figure \ref{F:hatexample}.  The $1$-Salem generalised line graphs of Theorem \ref{T:glg_result} are presented as $25$ infinite families and six sporadic graphs, but there are in fact $60$ non-isomorphic infinite families when we consider all the possible options of paths with snake's tongues attached.

Also, as a simple corollary to Theorem \ref{T:glg_result} we can easily note which of the graphs in Figures \ref{F:1salemglg1}--\ref{F:1salemglg3} are line graphs (rather than generalised line graphs) based on the characterization in Theorem \ref{T:glgcond}. These are the graphs where the edges can be partitioned into cliques rather than GCPs and there are $12$ infinite families and $3$ sporadic graphs.

\begin{figure}
\begin{center}
\begin{tabular}{cc}

\begin{tikzpicture}[scale=0.42, auto] 
\foreach \pos/\name in
{{(0,0)/a},{(1.414,0)/b},{(0.707,-0.707)/c},{(0.707,0.707)/d},{(-1,0)/e},{(2.414,0)/h}}
\node[vertex] (\name) at \pos {};
\foreach \edgetype/\source/ \dest in {pedge/a/b,pedge/a/c,pedge/a/e,pedge/a/d,pedge/b/c,pedge/b/d,pedge/b/h}
\path[\edgetype] (\source) -- (\dest);
\end{tikzpicture}

&

\begin{tikzpicture}[scale=0.42, auto] 
\foreach \pos/\name in
{{(0,0)/a},{(1.414,0)/b},{(0.707,-0.707)/c},{(0.707,0.707)/d},{(-1,0)/e},{(2.414,0)/h}}
\node[vertex] (\name) at \pos {};
\foreach \pos/\name in
{{(3.28,0.5)/ha},{(3.28,-0.5)/hb}}
\node[vertex] (\name) at \pos {};
\foreach \edgetype/\source/ \dest in {pedge/a/b,pedge/a/c,pedge/a/e,pedge/a/d,pedge/b/c,pedge/b/d,pedge/b/h,pedge/h/ha,pedge/h/hb}
\path[\edgetype] (\source) -- (\dest);
\end{tikzpicture}
\\
$G_{10}(1,1)$ & $G_{10}(1,\hat{1})$
\end{tabular}
\end{center}

\caption{An illustration of the hat convention.}
\label{F:hatexample}
\end{figure}
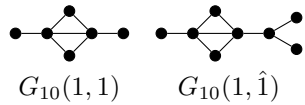

The following trivial extension of \cite[Theorem 3.4]{MS2005} can be used to show that all the graphs in Figures \ref{F:1salemglg1}--\ref{F:1salemglg3} are Salem graphs.

\begin{lemma}\label{L:1salem}
Suppose that $G$ is a non-cyclotomic non-bipartite graph containing a vertex $v$ such that the induced subgraph on $V\left(G\right)\backslash\left\{v\right\}$ is cyclotomic. Also suppose that $G$ is in the family of graphs with least eigenvalue greater than $-2$, then $G$ is a Salem graph.
\end{lemma}

To apply Lemma \ref{L:1salem} to the graphs in Figures \ref{F:1salemglg1}--\ref{F:1salemglg3}, we note that all the graphs in that figure are generalised line graphs (using the characterisation in Theorem \ref{T:glgcond}, for example), and in each case one readily spots a vertex $v$ whose deletion leaves a cyclotomic induced subgraph.

We now prove Theorem \ref{T:glg_result} using the structure given in Proposition \ref{L:hamlemma}; that is we will grow our graphs starting with the vertices in $M$, then $M \cup A$, then $M \cup A \cup H$. We quickly note a simple lemma that will be referred to frequently in the proof.

\begin{lemma} \label{L:k5k4}
A $1$-Salem graph $G$ may not contain an induced $K_5$. Also, if it contains an induced $K_4$, then only one of the four vertices in that $K_4$ may be attached to any other vertices in $G$ and the vertex we remove to make $G$ cyclotomic must be this distinguished vertex in $K_4$.
\end{lemma}

\begin{proof}
The first sentence is clear since if we remove any one of the vertices of $K_5$ we obtain a $K_4$ which is not cyclotomic. For the second sentence, removing a vertex $v$ of $K_4$ leaves a $K_3$ so if any of the other vertices of the $K_4$ were attached to any vertices of $G$, so would the $K_3$ be after removing $v$, and no connected supergraph of $K_3$ is cyclotomic.
\end{proof}

\subsection{$M$ --- the minimal graphs}

We begin by considering what our minimal graphs may look like.

\begin{proposition} \label{P:triangles}
All Salem generalized line graphs must contain a $K_3$.
\end{proposition}

\begin{proof}
We know from Theorem \ref{T:glgcond} that generalised line graphs are built from GCPs and it is easy to see that the only GCPs that do not contain a $K_3$ are $GCP\left(1,0\right)$, $GCP\left(2,0\right)$, $GCP\left(2,1\right)$, $GCP\left(3,1\right)$ and $GCP\left(4,2\right)$. These graphs are certainly all cyclotomic and, moreover, all the generalised line graphs that can be made using them (observing the rules in Theorem \ref{T:glgcond}) will be subgraphs of either $\tilde{A}_n$ or $\tilde{D}_n$. However, by definition Salem graphs are non-cyclotomic, so a Salem generalised line graph must include at least one GCP that contains a $K_3$.
\end{proof}

As a corollary to this we get that the minimal graphs we are after must be the three ways of attaching a single vertex to a $K_3$.

\begin{corollary} \label{C:minglgstheorem}
The minimal graphs with respect to the property of being a Salem generalized line graph are the three graphs in Figure \ref{F:mingraphs}.
\end{corollary}

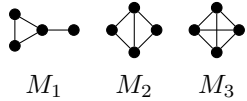
\begin{figure}[h]
\begin{center}
\begin{tabular}{ccc}
\begin{tikzpicture}[scale=0.42, auto]
\foreach \pos/\name in
{{(0,0)/a},{(-1,0)/b},{(-1.866,0.5)/d},{(-1.866,-0.5)/g}}
\node[vertex] (\name) at \pos {};
\foreach \pos/\name in
{{(-0.707,0.707)/z},{(-0.707,-0.707)/zz}}
\node[] (\name) at \pos {};
\foreach \edgetype/\source/ \dest in {pedge/a/b,pedge/b/g,pedge/b/d,pedge/d/g}
\path[\edgetype] (\source) -- (\dest);
\end{tikzpicture}
&
\begin{tikzpicture}[scale=0.42, auto]
\foreach \pos/\name in
{{(0,0)/a},{(-1.414,0)/c},{(-0.707,0.707)/e},{(-0.707,-0.707)/f}}
\node[vertex] (\name) at \pos {};
\foreach \pos/\name in
{{(-0.707,0.707)/z},{(-0.707,-0.707)/zz}}
\node[] (\name) at \pos {};
\foreach \edgetype/\source/ \dest in {pedge/a/e,pedge/a/f,pedge/c/e,pedge/c/f,pedge/e/f}
\path[\edgetype] (\source) -- (\dest);
\end{tikzpicture}
&
\begin{tikzpicture}[scale=0.42, auto]
\foreach \pos/\name in
{{(0,0)/a},{(-1.414,0)/c},{(-0.707,0.707)/e},{(-0.707,-0.707)/f}}
\node[vertex] (\name) at \pos {};
\foreach \pos/\name in
{{(-0.707,0.707)/z},{(-0.707,-0.707)/zz}}
\node[] (\name) at \pos {};
\foreach \edgetype/\source/ \dest in {pedge/a/c,pedge/a/e,pedge/a/f,pedge/c/e,pedge/c/f,pedge/e/f}
\path[\edgetype] (\source) -- (\dest);
\end{tikzpicture}
\\
$M_1$&$M_2$&$M_3$
\end{tabular}
\caption{The three minimal graphs in Corollary \ref{C:minglgstheorem}.}
\label{F:mingraphs}
\end{center}
\end{figure}

We now know what our minimal graphs look like but before proceeding we make one more observation. For $G\setminus \left\{v\right\}$ to be cyclotomic we must have $v\in M$; that is, the vertex we are removing to induce a cyclotomic graph must be one of the vertices of the minimal graph $M$.

\subsection{$A$ --- the adjacent vertices}

Growing our graph from $G\vert_M$ to $G\vert_{M \cup A}$ is the most difficult part, however the restriction to $1$-Salem generalized line graphs reduces this to a finite search, after some work. By Theorem \ref{T:glgcond} we have a highly-constrained structure: we cannot simply add vertices and edges anywhere. First let us consider the ways we can partition $M_1$, $M_2$ and $M_3$ into GCPs; this will reveal where we can add vertices. The generalised line graph $M_1$ can be seen uniquely as a $K_3$ and a $K_2$; it then has three vertices of maximal degree that are not already in two GCPs. And $M_3$ can be seen uniquely as a $K_4$, but by Lemma \ref{L:k5k4} we know that we can attach further vertices to only one of its original vertices. The graph $M_2$, however, is one of the seven graphs that has two non-isomorphic root multigraphs (see \cite[Theorem 2.3.4]{CRS2004}) so we need to consider both versions. Let $M_{2,1}$ be $M_2$ partitioned as a $K_3$ with two $K_2$'s attached and both joined at their other ends; here we only have one vertex that is of maximal degree and not already in two GCPs. Let $M_{2,2}$ be $M_2$ partitioned as a $GCP\left(4,1\right)$ where we then have two vertices of maximal degree.

The set of vertices $A$ are those in the graph $G\vert_{A \cup H}$ that are adjacent to a vertex in $M$.
By Theorem \ref{T:glgcond}, to grow from $G\vert_M$ to $G\vert_{M \cup A}$ we can:
\begin{itemize}
	\item expand a GCP to a larger one that contains the original one (taking care of the degrees of vertices);
	\item attach a new GCP to a vertex of maximal degree that is only in one GCP (attaching only at vertices of maximal degree);
	\item do both of the above.
\end{itemize}
Once we have done this we can then also add an edge between any two vertices of $A$ of maximal degree in their GCPs (in effect, adding a $K_2$) or take two maximal degree vertices of $A$ that are in separate GCPs and only in one GCP each and `merge' them together (connect two disconnected GCPs by making them share an available maximal degree vertex).

Lemma \ref{L:k5k4} helps here as we know that we need not consider any GCPs that contain a $K_5$. In fact, by studying the GCPs $G$ that have this property and looking at the induced graph $G \setminus \left\{v\right\}$ for each $v\in V\left(G\right)$ we find the fairly short list of GCPs that we can attach or expand to in Table \ref{tab:GCPs} below.

\begin{center}
\begin{table}[h]
\begin{tabular}{{m{3cm}m{1.7cm}m{4.4cm}}}
$GCP\left(2,0\right) = K_2$ &
\begin{tikzpicture}[scale=0.42, auto] 
\foreach \pos/\name in
{{(0,0)/a},{(1,0)/b}}
\node[vertex] (\name) at \pos {};
\foreach \edgetype/\source/ \dest in {pedge/a/b}
\path[\edgetype] (\source) -- (\dest);
\foreach \pos/\name in
{{(-0.5,0)/za},{(1.5,0)/zb}}
\node[] (\name) at \pos {};
\end{tikzpicture}
& 2 vertices of maximal degree\\
$GCP\left(3,1\right)$       &
\begin{tikzpicture}[scale=0.42, auto] 
\foreach \pos/\name in
{{(0,0)/a},{(0.866,0.5)/b},{(0.866,-0.5)/c}}
\node[vertex] (\name) at \pos {};
\foreach \edgetype/\source/ \dest in {pedge/a/b,pedge/c/a}
\path[\edgetype] (\source) -- (\dest);
\foreach \pos/\name in
{{(-0.567,0)/za},{(1.433,0)/zb}}
\node[] (\name) at \pos {};
\end{tikzpicture}
& 1 vertex of maximal degree\\
$GCP\left(3,0\right) = K_3$ &
\begin{tikzpicture}[scale=0.42, auto] 
\foreach \pos/\name in
{{(0,0)/a},{(0.866,0.5)/b},{(0.866,-0.5)/c}}
\node[vertex] (\name) at \pos {};
\foreach \edgetype/\source/ \dest in {pedge/a/b,pedge/b/c,pedge/c/a}
\path[\edgetype] (\source) -- (\dest);
\foreach \pos/\name in
{{(-0.567,0)/za},{(1.433,0)/zb}}
\node[] (\name) at \pos {};
\end{tikzpicture}
& 3 vertices of maximal degree\\
$GCP\left(4,1\right)$       &
\begin{tikzpicture}[scale=0.42, auto] 
\foreach \pos/\name in
{{(0,0)/a},{(-1.414,0)/c},{(-0.707,0.707)/e},{(-0.707,-0.707)/f}}
\node[vertex] (\name) at \pos {};
\foreach \edgetype/\source/ \dest in {pedge/a/e,pedge/a/f,pedge/c/e,pedge/c/f,pedge/e/f}
\path[\edgetype] (\source) -- (\dest);
\foreach \pos/\name in
{{(-1.707,0)/za},{(.293,0)/zb}}
\node[] (\name) at \pos {};
\end{tikzpicture}
& 2 vertices of maximal degree\\
$GCP\left(4,0\right) = K_4$ &
\begin{tikzpicture}[scale=0.42, auto] 
\foreach \pos/\name in
{{(0,0)/a},{(-1.414,0)/c},{(-0.707,0.707)/e},{(-0.707,-0.707)/f}}
\node[vertex] (\name) at \pos {};
\foreach \edgetype/\source/ \dest in {pedge/a/c,pedge/a/e,pedge/a/f,pedge/c/e,pedge/c/f,pedge/e/f}
\path[\edgetype] (\source) -- (\dest);
\foreach \pos/\name in
{{(-1.707,0)/za},{(.293,0)/zb}}
\node[] (\name) at \pos {};
\end{tikzpicture}
& 4 vertices of maximal degree\\
$GCP\left(5,2\right)$       &
\begin{tikzpicture}[scale=0.42, auto] 
\foreach \pos/\name in
{{(0,0)/a},{(-0.588,0.809)/b},{(-1.54,0.5)/c},{(-1.54,-0.5)/d},{(-0.588,-0.809)/e}}
\node[vertex] (\name) at \pos {};
\foreach \edgetype/\source/ \dest in {pedge/a/b,pedge/a/c,pedge/a/d,pedge/a/e,pedge/b/e,pedge/b/c,pedge/c/d,pedge/d/e}
\path[\edgetype] (\source) -- (\dest);
\foreach \pos/\name in
{{(-1.77,0)/za},{(.23,0)/zb}}
\node[] (\name) at \pos {};
\end{tikzpicture}
& 1 vertex of maximal degree\\
\end{tabular}
\caption{The GCPs that we can attach or expand to in $A$.}
\label{tab:GCPs}
\end{table}
\end{center}
Note that $GCP\left(4,2\right) = C_4$ is not included; when partitioned as $GCP\left(4,2\right)$ it has no vertices of maximal degree, so cannot be attached to anything and when partitioned as four $K_2$'s it has no vertices that are not already in two GCPs. Also, $GCP\left(3,1\right)$ and $GCP\left(5,2\right)$ are two of the seven graphs that have non-isomorphic root multi-graphs. However, when $GCP\left(5,2\right)$ is partitioned as two $K_3$'s and two $K_2$'s it has no vertices that are not already in two GCPs and when $GCP\left(3,1\right)$ is partitioned as two $K_2$'s one of its vertices will not be in $A$.

The process of going from $G\vert_M$ to $G\vert_{M \cup A}$ is then a finite one; we only have so many GCPs in the minimal graphs $M_1$, $M_{2,1}$, $M_{2,2}$ and $M_3$ to expand and only so many ways to attach these six GCPs to them. We also have a small number of cases where we can add in extra edges between vertices of $A$ or merge them. In working through all these combinations we discard a number of graphs that are not 1-Salem as they require more than one vertex to be removed to make them cyclotomic.

The list of $1$-Salem generalised line graphs $G\vert_{M \cup A}$ has been omitted for reasons of space.
There are $224$ that are distinct as graphs, although many of these graphs can arise in more than one way as $G\vert_{M \cup A}$.  The largest has $11$ vertices.
Initially these were all found by hand; the list was then checked by a computer search.

\subsection{$H$ --- the cyclotomic parts}

We now look at the set $H$ in Lemma \ref{L:hamlemma}. We can reduce our choices by observing that the only cyclotomic graphs that are also generalized line graphs are $\tilde{D}_n$ and $\tilde{A}_n=C_{n+1}$. However, we will show that $G\vert_H$ cannot contain cycles and can only contain subgraphs of $\tilde{D}_n$.  For $n>3$ we note that $\tilde{A}_n$ can be partitioned uniquely as $\left(n+1\right)$ $K_2$'s with each vertex in two GCPs, so we cannot simply attach vertices to it. The other option is then to expand a GCP to a larger one so that we now have vertices only in one GCP to attach to other things. The smallest case is to expand one of the $\left(n+1\right)$ $K_2$'s to a $K_3$. Clearly, this new vertex must be in $A$ rather than $H$ as the graph is no longer cyclotomic. However, in order to make the graph cyclotomic we must remove a vertex but all of the vertices are in $A$ or $H$ and we have shown that the vertex we remove must be in $M$. Cycles of length $3$ can be seen as both three $K_2$'s or one $K_3$ but by similar reasoning on the choice of vertex we are removing, we can exclude this case too. A similar argument holds again for cycles of length $4$ for both ways of partitioning its edges.

The graph $\tilde{D}_n$ can be uniquely partitioned as $\left(n-4\right)$ $K_2$'s with a $GCP\left(2,0\right)$ (or snake's tongue) at either end. However, in this graph each vertex of maximal degree within its GCP is already in two GCPs. If we remove one or both of the snake's tongues we are left with graphs we can work with---a path or a path with a snake's tongue on the end---each with at least one vertex of maximal degree that is only in one GCP. With this in mind we can think of $G\vert_H$ simply consisting of paths of any length, possibly with a snake's tongue on one end.

The final step in growing these graphs is to go from $G\vert_{M \cup A}$ to $G\vert_{M \cup A \cup H} =G$. To any vertices in $A$ of maximal degree and only in one GCP we can attach a single path of arbitrary length (remembering that Lemma \ref{L:k5k4} tells us that we can only attach things to one vertex of a $K_4$). If we so choose, we can join any two of these pendant paths together (equivalently, attach a path or arbitrary length to two different elements of $A$ of maximal degree that are each only in one GCP). Furthermore, on the end of any pendant paths we can include a snake's tongue. These graphs are then precisely the graphs in Figures \ref{F:1salemglg1}--\ref{F:1salemglg3} and Lemma \ref{L:1salem} tells us that they are all 1-Salem, completing this section of the proof.

We can consider letting the lengths of one or more pendant or internal paths tend to infinity. By \cite[Corollary 4.4]{MS2005} the corresponding sequence of graph Salem numbers converges to a Pisot number. By studying the different ways we can let paths tend to infinity, we can produce several explicit families of Pisot numbers.

\section{The $1$-Salem graphs represented in $E_8$}\label{S:E8}

\subsection{The root system and signed graph $E_8$}

There are various concrete descriptions of the root system $E_8$ (see \cite[Chapter 3]{CvL1991}, for example).
Here is one of them.
Let ${\bf e}_1$, \dots, ${\bf e}_8$ be an orthogonal basis for $\mathbb{R}^8$, with each ${\bf e}_i$ having length $\sqrt{2}$.  The root system $E_8$ contains $240$ vectors: the $16$ vectors $\pm{\bf e}_1$, \dots, $\pm{\bf e}_8$, and $224$ vectors of the form $(\pm{\bf e}_i\pm{\bf e}_j\pm{\bf e}_k\pm{\bf e}_l)/2$, where $ijkl$ is one of the $14$ strings $1234$, $1256$, $1278$, $1357$, $1368$, $1467$, $2358$, $2367$, $2457$, $2468$, $3456$, $3478$, $5678$.
The usual inner product of any two vectors from these $240$ is  one of $-2$, $-1$, $0$, $1$, or $2$; and all the vectors have length $2$.
A graph $G$ with adjacency matrix $A$ is said to be represented in $E_8$ if $A+2I$ is the Gram matrix of some subset of the vectors forming the root system $E_8$ (so vectors correspond to vertices, and adjacency of distinct vertices is given by the usual inner product of the corresponding vectors).

The \emph{Weyl group} $W(E_8)$ is the group of isometries of $E_8$ generated by the involutions ${\bf v}\mapsto {\bf v} - ({\bf v}\cdot{\bf v}_i){\bf v}_i$ for each of the $240$ vectors ${\bf v}_i$ (in fact there are only $120$ of these involutions, since ${\bf v}_i$ and $-{\bf v}_i$ yield the same involution).

\subsection{Fast backtracking in $E_8$}\label{SS:backtracking}

Let $P$ be a property of graphs such that if $G$ has property $P$ then so does any induced subgraph.  (Such a property is called \emph{hereditary}.)
The following natural process, familiar to computational graph-theorists, provides a computationally-efficient method for searching for all graphs that: (a) are represented in $E_8$; and (b) have property $P$.

Assign an ordering to the vectors of $E_8$: ${\bf v}_1$, \dots, ${\bf v}_{240}$.
A naive search for all graphs having property $P$ that are represented in $E_8$ would proceed as follows.
Start with an empty list of vectors; look through the vectors in order, adding them to the list so long as the vectors in the list represent a graph that has property $P$; once stuck, backtrack.
The number of graphs represented in $E_8$ is so large that this process is only practical for extremely restrictive properties $P$.  To improve matters, we exploit the isometries of $E_8$.

Choose a $S\subseteq W(E_8)$.  When considering if it is possible to add a new vector to ones list, apply each isometry in $S$ to the potential enlarged list, and reject the new vector if any of the permuted lists appears earlier in the lexicographical ordering derived from our ordering of the vectors.
Larger sets $S$ make each update more expensive, but potentially reduce the backtracking by pruning certain branches at an early stage.

We used a few other programming tricks to speed up the process further, but the main idea was to exploit isometries of $E_8$.  In practice we used a set $S$ of size $162$, including the $120$ involutions above, and $42$ carefully selected others.





\subsection{$1$-Salem graphs represented in $E_8$}

The property of being a $1$-Salem graph is not hereditary, but the property of being either $1$-Salem or cyclotomic is.
We used the fast backtracking approach of \S\ref{SS:backtracking} to find all such graphs represented in $E_8$, and then weeded out those that were cyclotomic, and also those that were generalised line graphs (and hence covered by part (ii) of Theorem \ref{T:summary}).
Table \ref{tab:notglgs} summarises the results, showing the number of non-bipartite connected $1$-Salem graphs that are not generalised line graphs by the number of vertices.
The largest examples have $11$ vertices; they are shown in Figure \ref{fig:11}.
These computations complete the final part of the proof of Theorem \ref{T:summary}, which we state here as a theorem.

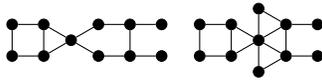
\begin{figure}
\begin{tabular}{cc}
\begin{tikzpicture}[scale=0.42, auto]
\foreach \pos/\name in
{{(-2.732,1)/a},{(-2.732,0)/b},{(-1.732,1)/c},{(-1.732,0)/d},{(-0.866,0.5)/e},{(0,1)/f},{(0,0)/g},{(1,1)/h},{(1,0)/i},{(2,1)/j},{(2,0)/k}}
\node[vertex] (\name) at \pos {};
\foreach \edgetype/\source/ \dest in {pedge/a/b,pedge/a/c,pedge/b/d,pedge/c/d,pedge/c/e,pedge/d/e,pedge/e/f,pedge/e/g,pedge/f/h,pedge/g/i,pedge/h/i,pedge/h/j,pedge/i/k}
\path[\edgetype] (\source) -- (\dest);
\foreach \pos/\name in
{{(-0.866,1.5)/za},{(-0.866,-0.5)/zb}}
\node[] (\name) at \pos {};
\end{tikzpicture}
&
\begin{tikzpicture}[scale=0.42, auto]
\foreach \pos/\name in
{{(-2.732,1)/a},{(-2.732,0)/b},{(-1.732,1)/c},{(-1.732,0)/d},{(-0.866,0.5)/e},{(0,1)/f},{(0,0)/g},{(1,1)/h},{(1,0)/i},{(-0.866,1.5)/j},{(-0.866,-0.5)/k}}
\node[vertex] (\name) at \pos {};
\foreach \edgetype/\source/ \dest in {pedge/a/b,pedge/a/c,pedge/b/d,pedge/c/d,pedge/c/e,pedge/d/e,pedge/e/j,pedge/e/f,pedge/e/g,pedge/e/k,pedge/j/f,pedge/f/g,pedge/g/k,pedge/f/h,pedge/g/i}
\path[\edgetype] (\source) -- (\dest);
\foreach \pos/\name in
{{(-0.866,1.5)/za},{(-0.866,-0.5)/zb}}
\node[] (\name) at \pos {};
\end{tikzpicture}
\end{tabular}
\caption{The two $11$-vertex $1$-Salem graphs that are not generalised line graphs.}\label{fig:11}
\end{figure}

\begin{theorem}\label{T:e8_result}
There are $377$ non-bipartite $1$-Salem graphs that are not generalised line graphs.
The numbers of vertices for these graphs range between $6$ and $11$; the number of graphs for each of these numbers of vertices is shown in Table \ref{tab:notglgs}.
\end{theorem}

\begin{table}[h]
\[
\begin{array}{|l||r|r|r|r|r|r|}\hline
\textrm{Number of vertices} & 6 & 7 & 8 & 9 & 10 & 11 \\ \hline
\textrm{Number of graphs} & 10 & 43 & 111 & 153 & 58 & 2 \\ \hline
\end{array}
\]
\caption{The number of non-bipartite $1$-Salem graphs that are not generalised line graphs.}
\label{tab:notglgs}
\end{table}

\section{Acknowledgement}
We are grateful to Jonathan Cooley, who verified several of our computations independently, hugely increasing our confidence in their accuracy.


\begin{thebibliography}{10}

\bibitem{Bo1979} B.~Bollob\'{a}s, \emph{Graph theory: an introductory course}, Graduate Texts in Mathematics \textbf{63}, Springer, New York, 1979.

\bibitem{B1977} D.W.~Boyd, {Small Salem numbers}, \emph{Duke Math.~J.} \textbf{44}(2) (1977), 315--328.

\bibitem{CGSS1976} P.J.~Cameron, J.M.~Goethals, J.J.~Seidel, E.E.~Shult, {Line graphs, root systems, and elliptic geometry}, \emph{J.~Algebra} \textbf{43} (1976), 305--327.

\bibitem{CvL1991} P.J.~Cameron, J.H.~van Lint, \emph{Designs, graphs, codes and their links}, Cambridge University Press, Cambridge, 1991.

\bibitem{Ca1829} A.L.~Cauchy, {Sur l'\'{e}quation \`{a} l'aide de laquelle on d\'{e}termine les in\'{e}galiti\'{e}s s\'{e}culaires des mouvements des plan\`{e}tes}, Oeuvres compl\`{e}tes, Ii\`{e}eme S\'{e}rie, \textbf{9}, 174--195, Gauthier-Villars, 1829.

\bibitem{CKST2011} T.~Chung, J.~Koolen, Y.~Sano, T.~Taniguchi, {The non-bipartite integral graphs with spectral radius three}, \emph{Linear Algebra and its Applications} \textbf{435}(10) (2011), 2544--2559.

\bibitem{CDG1982} D.~Cvetkovi\'{c}, M.~Doob, I.~Gutman, {On graphs whose eigenvalues do not exceed $\sqrt{2+\sqrt{5}}$}, \emph{Ars Comb.} \textbf{14} (1982), 225--239.

\bibitem{CDS1981} D.~Cvetkovi\'{c}, M.~Doob, S.~Simi\'{c}, {Generalised line graphs}, \emph{J.~Graph Theory} \textbf{5} (1981), 385--399.

\bibitem{CRS2004} D.~Cvetkovi\'{c}, P.~Rowlinson, S.~Simi\'{c}, \emph{Spectral generalizations of line graphs: on graphs with least eigenvalue $-2$}, London Mathematical Society Lecture Note Series \textbf{314}, Cambridge University Press, Cambridge, 2004.

\bibitem{CW1992} J.W.~Cannon, Ph.~Wagreich, {Growth functions of surface groups}, \emph{Math.~Ann.} \textbf{293} (1992), 239--257.

\bibitem{Do2008} E.~Dobrowolski, {A note on integer symmetric matrices and Mahler's measure}, \emph{Canad.~Math.~Bull.} \textbf{51}(1) (2008), 57--59.

\bibitem{Es1992} D.R.~Estes, {Eigenvalues of symmetric integer matrices}, \emph{J.~Number Theory} \textbf{42} (1992), 292--296.

\bibitem{GM1978} C.~Godsil, B.~McKay, {A new graph product and its spectrum}, \emph{Bull.~Australian Math.~Soc.} \textbf{18}(01) (1978), 21--28.

\bibitem{GR2001} C.~Godsil, G.~Royle, \emph{Algebraic Graph Theory}, Springer, New York, 2001.

\bibitem{Ha1969} F.~Harary, \emph{Graph Theory}, Addison-Wesley, 1969.

\bibitem{Ho1972} A.J.~Hoffman, {On limit points of spectral radii of non-negative symmetric matrices}, in \emph{Graph Theory and its Applications}, Lecture Notes in Math.~303,  165--172, Springer, Berlin, 1972.

\bibitem{Ho1975} A.J.~Hoffman, {Eigenvalues of graphs}, in \emph{Studies in Graph Theory}, 225--245, ed.~D.R.~Fulkerson, Math.~Assoc.~Amer., Washington, 1975.

\bibitem{HS1975} A.J.~Hoffman, J.H.~Smith, {On the spectral radii of topologically equivalent graphs}, in \emph{Recent advances in graph theory}, 273--281, ed.~M.~Fiedler, Academia Praha, 1975.

\bibitem{Kr1857} L.~Kronecker, {Zwei s\"{a}tse \"{u}ber gleichungen mit ganzzahligen coefficienten}, \emph{J.~Reine Angew.~Math.} \textbf{53} (1857), 173--175.

\bibitem{L1999} P.~Lakatos, {On spectral radius of Coxeter transformation of trees}, \emph{Publicationes Mathematicae Debrecen}, \textbf{54} (1999), 181-187.

\bibitem{L2001} P.~Lakatos, {Salem numbers, PV numbers and spectral radii of Coxeter transformations}, \emph{Comptes Rendus Math.~de L'Acad.~des Sciences La Soc.~Royale du Canada}, \textbf{23} (2001), 71--77.

\bibitem{L2010} P.~Lakatos, {Salem numbers defined by Coxeter transformation}, \emph{Linear Algebra and its Appl.}, \textbf{432} (2010), 144--154.

\bibitem{M2000} J.~McKee, {Families of Pisot numbers with negative trace}, \emph{Acta Arith.}, \textbf{93}(4), (2000),  373--385.

\bibitem{MRS1999} J.~McKee, C.~Smyth, P.~Rowlinson, {Salem numbers and Pisot numbers from stars}, in \emph{Number Theory in Progress}, vol.~1 (Zakopane-Kocielisko, 1997), ed.~K.~Gy\H{o}ry, H.~Iwaniec, J.~Urbanowicz, 309--319, de Gruyter, Berlin, 1999.

\bibitem{McK2010} J.~McKee, {Small-span characteristic polynomials of integer symmetric matrices}, in {Algorithmic Number Theory}, ed.~G.~Hanrot, F.~Morain, E.~Thom\'{e}, $9$th International Symposium, ANTS-IX, Nancy, France, Springer Lecture Notes in Computer Science \textbf{6197} (2010), 270--284.

\bibitem{MS2004} J.~McKee, C.~Smyth, {Salem numbers of trace $-2$, and traces of totally positive algebraic integers}, in {Algorithmic Number Theory}, ed.~D.A.~Buell, $6$th International Symposium, ANTS-VI, Burlington, USA, Springer Lecture Notes in Computer Science \textbf{3076} (2004), 327--337.

\bibitem{MS2005} J.~McKee, C.~Smyth, {Salem numbers, Pisot numbers, Mahler measure, and graphs}, \emph{Experimental Mathematics} \textbf{14} (2005), 211--229.

\bibitem{MS2007} J.~McKee, C.~Smyth, {Integer symmetric matrices having all their eigenvalues in the interval $[-2,2]$}, \emph{J.~Algebra} \textbf{317} (2007), 260--290.

\bibitem{MS2011} J.~McKee, C.~Smyth, {Integer symmetric matrices of small spectral radius and small Mahler measure}, \emph{Int.~Math.~Research Notices} (2011); doi: 10.1093/imrn/rnr011.

\bibitem{M2002} C.T.~McMullen, {Coxeter groups, Salem numbers and the Hilbert metric}, \emph{Publ.~Math.~Inst.~de Hautes Etudes Scientifiques} \textbf{95} (2002), 151--183.

\bibitem{P1993} W.~Parry, {Growth series of Coxeter groups and Salem numbers}, \emph{J.~Algebra} \textbf{154} (1993), 406--415.

\bibitem{PI1994} V.~Prasolov, S.~Ivanov, \emph{Problems and theorems in linear algebra}, American Mathematical Society, 1994.

\bibitem{RSV1981} S.B.~Rao, N.M.~Vinghi, K.S.~Vijayan, {The minimal forbidden subgraphs for generalized line-graphs}, in \emph{Combinatorics and Graph Theory, Proc.~Symp.~Indian Statistical Institute, Calcutta, 25--29 February 1980}, edited by S.B.~Rao, 459--472, Lecture Notes in Math.~\textbf{855}, Springer, Berlin, 1981.

\bibitem{Ro1962} R.M.~Robinson, {Intervals containing infinitely many sets of conjugate algebraic integers}, \emph{Mathematical analysis and related topics: essays in honor of George P\'{o}lya}, 305--315, Stanford, 1962.

\bibitem{Sc1918} I.~Schur, {\"{U}ber die Verteilung der Wurzeln bei gewissen algebraischen Gleichungen mit ganzzahligen Koeffizienten}, \emph{Math.~Z.} \textbf{1} (1918), 377--402.

\bibitem{Sm1970} J.H.~Smith, {Some properties of the spectrum of a graph}, in \emph{Combinatorial structures and their applications (Proc.~Calgary Internat.~Conf., Calgary, Alta., 1969)}, edited by R.~Guy, H.~Hanani, H.~Saver, J.~Sch\"{o}nheim, 403--406, Gordon and Breach, New York,  1970.

\bibitem{S2008} R.~Stekolshchik, \emph{Notes on Coxeter transformations and the McKay correspondence}, Springer, Berlin Heidelberg, 2008.

\bibitem{Z1989} Y.~Zhang, {Eigenvalues of Coxeter transformations and the structure of regular components of an Auslander-Reiten quiver}, \emph{Comm.~Algebra} \textbf{17}(10) (1989), 2347--2362.

\end{thebibliography}
\end{document}